\documentclass[11pt]{article}
\usepackage[all,2cell]{xy} \UseAllTwocells \SilentMatrices
\usepackage{latexsym,amsfonts,amssymb}
\usepackage{amsmath,amsthm,amscd}
\usepackage{hyperref,psfrag}
\usepackage{diagcat}
\usepackage{color}
\usepackage{etoolbox}
\usepackage[dvips]{epsfig}
\usepackage{psfrag}

\usepackage{graphicx}
\usepackage[centertableaux]{ytableau}
\usepackage{a4wide}
\usepackage{geometry}

\usepackage{epigraph,wrapfig}

\normalfont\upshape

\usepackage{fancyheadings}
\theoremstyle{definition}
\newtheorem{thm}{Theorem}[section]
\newtheorem{cor}[thm]{Corollary}

\newtheorem{lem}[thm]{Lemma}
\newtheorem{rem}[thm]{Remark}
\newtheorem{prop}[thm]{Proposition}
\newtheorem{defn}[thm]{Definition}
\newtheorem{example}[thm]{Example}

\numberwithin{equation}{section}

\usepackage{bbm}

\def\N{{\mathbbm N}}

\def\Z{{\mathbbm Z}}
\def\Q{{\mathbbm Q}}

\def\F{{\mathbbm F}}

\def\1{{\mathbbm{1}}}

\newcommand{\Hom}{{\rm Hom}}
\newcommand{\HOM}{{\rm HOM}}
\renewcommand{\to}{\rightarrow}

\newcommand{\End}{{\rm End}}
\newcommand{\END}{{\rm END}}

\def\dif{\partial}

\def\lra{{\longrightarrow}}
\def\Id{\mathrm{Id}}
\def\mc{\mathcal}
\def\mf{\mathfrak}
\def\shuffle{\,\raise 1pt\hbox{$\scriptscriptstyle\cup{\mskip
               -4mu}\cup$}\,}
\newcommand{\refequal}[1]{\xy {\ar@{=}^{#1}
(-1,0)*{};(1,0)*{}};
\endxy}

\newcommand{\sym}{\mathrm{Sym}}

\newcommand{\pol}{\mathrm{Pol}}
\newcommand{\s}{\mathcal{S}}

\newcommand{\mH}{\mathrm{H}} %cohomology

\newstrandstyle{thk1}{type=string,color=green,width=3pt} %%Thick strings
\newstrandstyle{thn1}{type=string,color=green,width=1pt} 
\newstrandstyle{thn2}{type=string,color=blue,width=1pt}

\title{A categorification of a quantum Frobenius map}
\author{You Qi}

\date{\today}

\begin{document}
%
% ==============================================================================

\maketitle

\begin{abstract}
A quantum Frobenius map a la Lusztig for $\mathfrak{sl}_2$ is categorified at a prime root of unity.
\end{abstract}

\setcounter{tocdepth}{2} \tableofcontents

%%%%%%%%%%%%%%%%%%%%%%%%%%%%%%%%%
%%%%%%%%%%%%%%%%%%%%%%%%%%%%%%%%%
%%%%%%%%   INTRODUCTION  %%%%%%%%
%%%%%%%%%%%%%%%%%%%%%%%%%%%%%%%%%
%%%%%%%%%%%%%%%%%%%%%%%%%%%%%%%%%

\section{Introduction}
At a root of unity, the quantum Frobenius map was introduced by Lusztig \cite{LusModularRepQGp} as a characteristic zero analogue of the (Hopf dual of the) classical Frobenius map for algebraic groups in positive characteristic. It plays an essential role in the study of representation theory for quantum groups at a root of unity, and serves as a basic building block for the celebrated Kazhdan-Lusztig equivalence \cite{KaLuIV}. We refer the reader to the monograph of Lusztig's \cite{Lus4} for some fundamental applications.

Since then, important progress has been made towards   understanding Lusztig's powerful construction from a more geometric or categorical point of view. Notably the work of Arkhipov-Gaitsgory~\cite{ArGa} and Arkhipov-Bezrukavnikov-Ginzburg~\cite{ABG}, introduces, as a biproduct of their study, a geometric context to exhibit and utilize the quantum Frobenius map. Another beautiful and more categorical interpretation via the Hall algebra of Ringel-Lusztig has been devoloped by McGerty \cite{McG}.

The current work is aimed at providing yet another approach to understanding the quantum Frobenius map via categorification, in the simplest case of $\mathfrak{sl}_2$ at a prime root of unity. We hope this method will be prove applicable in the future, when one studies the categorical representation theory  $2$-Kac-Moody algebras, in the sense of Khovanov-Lauda-Rouquier \cite{KL1,KL2, KL3, Rou2}, at a prime root of unity. 

Below we give a brief summary of the content of the paper. We will assume that the reader has some basic familiarity of the work \cite{EQ1, EQ2, KLMS}, but we will recall all necessary facts from the references when needed.

 In Section \ref{sec-sym-polynomial}, we recall some constructions in the theory of $p$-differential graded algebras (p-DG algebras). We introduce the ``slash-zero formal" $p$-DG algebras, which are analogues of formal DG algebras. As in the ordinary DG case, slash-zero formality allows one to compute algebraic invariants of a $p$-DG algebra, such as its Grothendieck group, by passing to its slash cohomology $p$-DG algebra. Next, we give a family of examples of slash-zero formal $p$-DG algebras coming from symmetric polynomials equipped with a $p$-differential. These examples are studied in \cite{KQ,EQ1,EQ2}, and have played central\footnote{These polynomial algebras are literally the center of the categories used to categorify quantum $\mf{sl}_2$.} roles in categorifying quantum $\mf{sl}_2$, both finite and infinite dimensional versions, at a prime root of unity.

In Section \ref{sec-Gr-bim}, we will recall the half versions (upper or lower nilpotent part) of quantum $\mf{sl}_2$ at a generic value and at a prime root of unity (in the sense of Lusztig), as well as the definition of the quantum Frobenius map. Then we will review briefly the categorification of the one-half of the big quantum group \cite{EQ2} by the $p$-DG category of symmetric polynomials $\mc{D}(\sym)$. A (weak) categorification of the quantum Frobenius map at a $p$th root of unity is then constructed (\ref{thm-weak-categorification-half-sl2}) by restricting to a full subcategory $\mc{D}_{(p)}(\sym)$ inside $\mc{D}(\sym)$, which is generated by symmetric polynomials whose number of variables are divisible by $p$. The effect of this restriction map is readily read off on the Grothendieck group level to be the quantum Frobenius map. The next subsection is then devoted to enhance the weak categorification result by analyzing, diagrammatically, the $2$-morphism $p$-DG algebras that control the full subcategory $\mc{D}_{(p)}(\sym)$. The $2$-morphism $p$-DG algebras turn out to be slash-zero formal (Theorem \ref{thm-nil-Hecke-holds}), and are quasi-isomorphic to nilHecke algebras with the zero differential. The verification of the slash-zero formality constitutes the technical core of this Section.

The last Section \ref{sec-full-sl2} is dedicated to performing a ``doubling construction'' of the ``half'' categorification Theorem \ref{thm-nil-Hecke-holds}, thus giving rise to a categorified quantum Frobenius for an idempotented version of $\mf{sl}_2$ at a prime root of unity. We will achieve this by using the biadjuntions of \cite{KLMS} equipped with $p$-differentials, which has been studied in \cite{EQ2}.  More precisely, we investigate the full $p$-DG $2$-subcategory $\mc{D}_{(p)}(\dot{\mc{U}})$ inside the $p$-differential graded thick calculus $\mc{D}(\dot{\mc{U}})$ generated by $\mc{E}^{(p)}\1_{kp}$ and $\mc{F}^{(p)}\1_{kp}$. The main Theorem \ref{thm-main} shows that the restriction functor from $\mc{D}(\dot{\mc{U}})$ to $\mc{D}_{(p)}(\dot{\mc{U}})$ indeed categorifies the quantum Frobenius for $\mathfrak{sl}_2$. The proof of the Theorem uses the simplified relation-checking criterion of Brundan \cite{Brundan2KM} for Khovanov-Lauda-Rouquier's $2$-Kac-Moody categorification theorem. In the course of the proof, we also obtain a reduction result (Theorem \ref{thm-4-generators}) which shows that $\mc{D}(\mc{\dot{U}})$ can be generated by $1$-morphisms of the form $\mc{E}\1_n$, $\mc{E}^{(p)}\1_n$, $\mc{F}\1_n$, $\mc{F}^{(p)}\1_n$, with $n$ ranging over $\Z$. This categorifies the well-known fact that the big quantum group at a prime root of unity does not need all divided power elements to generate it as an algebra, but only requires undivided powers and the $p$th divided powers.

We conclude this introduction (see also Remark \ref{rmk-wrong-way-map}) by pointing out that the induction functor from $\mc{D}_{(p)}(\sym)$ (resp.~$\mc{D}_{(p)}(\mc{\dot{U}})$) to the ambient category $\mc{D}(\sym)$ (resp. $\mc{D}(\dot{\mc{U}})$)  is, in hindsight, a categorification of the canonical section for the quantum Frobenius map. 
A more ``correct'' categorical lifting of the quantum Frobenius should first be a categorification of the quotient of big quantum $\mathfrak{sl}_2$ by the ideal generated by the small quantum group, and then identifying the $p$-DG quotient category with categorified $\mathfrak{sl}_2$ with the zero differential.  We would like to develop a proper context for studying this ``$p$-DG quotient'' construction, which parallels Drinfeld's \emph{DG quotient of DG categories} \cite{DrDG}, in the framework of hopfological algebra \cite{Hopforoots, QYHopf}. This approach would hopefully be generalizable to quantum groups for higher rank Lie algebras.

\paragraph{Acknowledgments.} The author is very grateful to Rapha\"{e}l Rouquier for the inspiring motivational discussions on the project. He thanks Joshua Sussan and the referee for carefully reading the manuscript and suggesting many helpful corrections. He would also like to thank Igor Frenkel and Mikhail Khovanov for their constant support and encouragement. 

%%%%%%%%%%%%%%%%%%%%%%%%%%%%%%%%%%%%%%%%%
%%%%%%%%%%%%%%%%%%%%%%%%%%%%%%%%%%%%%%%%%
%%%%%%%%   Grassmannian Cohomology %%%%%%
%%%%%%%%%%%%%%%%%%%%%%%%%%%%%%%%%%%%%%%%%
%%%%%%%%%%%%%%%%%%%%%%%%%%%%%%%%%%%%%%%%%

\section{Symmetric polynomials with a \texorpdfstring{$p$}{p}-differential} \label{sec-sym-polynomial}
In this Section we will recall  $p$-differential graded ($p$-DG) algebras and their hopfological algebra, as discussed in \cite{Hopforoots, QYHopf}. The example of symmetric polynomials with a differential found in \cite{KQ} is then introduced, which will play a key role in categorifying the quantum Frobenius map later.

\paragraph{Conventions.}Once and for all, fix $\Bbbk$ to be a field of finite characteristic $p>0$. Tensor products of vector spaces over $\Bbbk$ will be written as $\otimes$ without decoration. Throughout, we take the set of natural numbers $\N$ to contain the zero element.

\subsection{Slash-zero formal $p$-DG algebras} 
\paragraph{Slash cohomology.} To begin with, we recall the definition of $p$-differential graded algebras and slash cohomology. The reader is referred to \cite[Section 2]{KQ} for a more full summary and \cite{Hopforoots, QYHopf} for further details.

\begin{defn}\label{def-p-DGA} A \emph{$p$-differential graded algebra}, or a $p$-DG algebra for short, is a $\Z$-graded $\Bbbk$-algebra equipped with a degree-two endomorphism $\dif_A:A\lra A$, such that the \emph{$p$-nilpotency condition} and the \emph{Leibnitz rule}
\[
\dif_A^p(a)=0, \quad \quad \dif_A(ab)=\dif_A(a)b+a\dif_A(b)
\] 
are satisfied for any elements $a,b\in A$.

A \emph{left $p$-DG module $M$ over $A$} is a $\Z$-graded left $A$-module together with a degree-two endomorphism $\dif_M:M\lra M$, such that, for any $a\in A$ and $x\in M$, 
\[
\dif_M^p(x)=0, \quad \quad \dif_{M}(ax)= \dif_A(a)x+a\dif_M(x).
\]
A \emph{morphism} of $p$-DG modules is an $A$-linear module map that commutes with differentials. 
\end{defn}

One can similarly define the notion of a right $p$-DG module over a $p$-DG algebra. When the context is clear, we will usually drop the subscripts from the differentials for $p$-DG algebras and modules. 

As the simplest example, the ground field $\Bbbk$ equipped with the zero differential is a $p$-DG algebra. A left or right $p$-DG module over $\Bbbk$ is called a \emph{$p$-complex}. 

We next define the \emph{slash cohomology} groups of a $p$-complex following \cite[Section 2.1]{KQ}. 

\begin{defn}\label{def-slash-cohomology}
Let $U$ be a $p$-complex over $\Bbbk$. For each $k\in \{0,\dots, p-2\}$, the $k$th \emph{slash cohomology} of $U$ is the graded vector space
\[
\mH_{/k}(U):=\dfrac{\mathrm{Ker}(\dif^{k+1})}{\mathrm{Im}(\dif^{p-k-1})+\mathrm{Ker}(\dif^k)}.
\] 
Elements in $\mathrm{Ker}(\dif^{k+1})$ will be called \emph{$k$-cocycles}, while those in $\mathrm{Im}(\dif^{p-k-1})$ will be referred to as \emph{$k$-coboundaries}.

The vector space $\mH_{/k}(U)$ inherits a $\Z$-grading coming from that of $U$. The (total) \emph{slash cohomology} of $U$ is the sum
\[
\mH_{/}(U):=\bigoplus_{k=0}^{p-2}\mH_{/k}(U).
\]
\end{defn}

It is easy to see that the differential on $U$ induces a $\Bbbk$-linear map
\[
\dif: \mH_{/k}(U)\lra \mH_{/k-1}(U),
\]
so that $\mH_{/}(U)$ collects to be a $p$-complex that satisfies $\dif^{p-1}|_{\mH_{/}(U)} = 0$. As a $p$-complex, we have a decomposition 
\[
U \cong \mH_{/}(U)\oplus P(U),
\]
where $P(U)$ is a \emph{contractible} $p$-complex. A \emph{contractible} $p$-complex is a $p$-complex that is isomorphic to a direct sum of $p$-dimensional complexes of the form
\[
0\lra \underline{\Bbbk}\stackrel{\cdot 1}{\lra} \Bbbk\stackrel{\cdot 1}{\lra} \cdots  \stackrel{\cdot 1}{\lra}\Bbbk\lra 0,
\]
where the underlined $\underline{\Bbbk}$ may start in any $\Z$ degree. The identity morphism of this complex can be written as
\[
\mathrm{Id} = \sum_{i=0}^{p-1}\dif^i \circ h \circ \dif^{p-i},
\]
where $h$ is the $\Bbbk$-linear map which identifies the right-most copy of $\Bbbk$ with the initial copy while being zero everywhere else:
\[
\begin{gathered}
\xymatrix{
0 \ar[r] & \underline{\Bbbk} \ar[r]^\dif\ar[d]_{\mathrm{Id}} & \Bbbk \ar[r]^\dif \ar[d]_{\mathrm{Id}}&  \cdots \ar[r]^\dif & \Bbbk \ar[dlll]|-{h} \ar[r] \ar[d]_{\mathrm{Id}}& 0\\
0 \ar[r] & \underline{\Bbbk} \ar[r]_\dif & \Bbbk \ar[r]_\dif &  \cdots \ar[r]_\dif & \Bbbk \ar[r] & 0
}
\end{gathered} \ .
\]
This definition of a contractible complex is analogous to the usual definition of a contractible complex of (super) vector spaces. In particular, in characteristic two, being a direct sum of complexes of the form
\[
0\lra \Bbbk \stackrel{\cong}{\lra} \Bbbk \lra 0
\]
is equivalent to having the identity morphism be nulhomotopic.

More generally, let $M,N$ be $p$-DG modules over some $p$-DG algebra $A$. A morphism of $p$-DG modules $f:M\lra N$ is called \emph{null-homotopic} if there is a degree-$(2-2p)$ $A$-linear map $h:M\lra N$ such that
\[
f=\sum_{i=0}^{p-1}\dif_N^{i}\circ h \circ \dif_M^{p-1-i}.
\]

Unlike the usual cohomology of a DG algebra, the slash cohomology space of a $p$-DG algebra need not satisfy the usual K\"{u}nneth-type equation
\[
\mH_{/}(A\otimes A)\cong \mH_{/}(A)\otimes \mH_{/}(A).
\] 
However, such a formula holds if the slash-cohomology of $A$ is particularly simple. Namely, when $\mH_/(A)=\mH_{/0}(A)$, or equivalently, when $\mH_/(A)$ is a $p$-complex with the trivial (zero) differential, the above tensor product formula holds. More generally, we have the following.

\begin{lem}\label{lemma-Kunneth}
Let $A$ be a $p$-DG algebra whose slash cohomology $\mH_/(A)$ is a $p$-complex with the zero differential, and let $M$ be a $p$-DG module. Then the natural action map
$
m: A\otimes M\lra M
$
induces an associative action on slash cohomology
\[
m_/: \mH_/(A)\otimes \mH_/(M)\lra \mH_/(M).
\]
\end{lem}
\begin{proof}
As $p$-complexes, we have
\[
A\cong \mH_/(A)\oplus P(A),\quad \quad M\cong \mH_/(M)\oplus P(M),
\]
which implies
\[
A\otimes M\cong \mH_/(A)\otimes \mH_/(M) \oplus P(A)\otimes \mH_/(M) \oplus \mH_/(A)\otimes P(M) \oplus P(A)\otimes P(M).
\]
If $\mH_/(A)$ is a trivial $p$-complex, then $\mH_/(A)\otimes \mH_/(M)$ is isomorphic to a direct sum of $\mathrm{dim}(\mH_/(A))$ copies of $\mH_/(M)$, and thus does not contain any contractible $p$-complex. The remaining three direct summands are all contractible, since a tensor product of any $p$-complex with a contractible $p$-complex is contractible. On the other hand, we have
\[
A\otimes M \cong \mH_/(A\otimes M)\oplus P(A\otimes M).
\]
As the abelian category of $p$-complexes is Krull-Schmidt (after all, it consists of graded modules over the finite-dimensional graded Hopf algebra $\Bbbk[\dif]/(\dif^p)$), we obtain an isomorphism of non-contractible $p$-complex summands
\[
\mH_/(A\otimes M)\cong \mH_/(A)\otimes \mH_/(M).
\]
The result then follows just as in the usual DG algebra case.
\end{proof}

\paragraph{Derived categories.} Let $A$ be a $p$-DG algebra. The \emph{$p$-DG homotopy category $\mc{H}(A,\dif)$} (or $\mc{H}(A)$ for short) is the categorical quotient of the abelian category of $p$-DG modules over $A$ by the ideal of null-homotopic morphisms. The \emph{$p$-DG derived category} $\mc{D}(A,\dif)$ (or $\mc{D}(A)$ for short) of $p$-DG modules is obtained from $\mc{H}(A,\dif)$ by inverting \emph{quasi-isomorphisms}. Recall that a morphism $f:M\lra N$ between two $p$-DG modules is a quasi-isomorphism if 
$$f_/: \mH_/(M)\lra \mH_/(N)$$
is an isomorphism of $p$-complexes. The homotopy category $\mc{H}(A)$ and the derived category $\mc{D}(A)$ serve as categorical invariants associated to any $p$-DG algebra $A$. We refer the reader to \cite[Section 3--4]{QYHopf} for some basic structural results about these triangulated categories.

An equivalent construction of the derived category can be obtained by restricting, in the homotopy category, to a smaller class of $p$-DG modules which are considered to be cofibrant objects. A $p$-DG module $P$ is called \emph{cofibrant} if, given any surjective quasi-isomorphism of $p$-DG modules $f:M\lra N$, the induced map of $p$-complexes
\begin{equation}
\HOM_A(P,M) \lra \HOM_A(P,N), \quad h\mapsto f\circ h
\end{equation}
is a quasi-isomorphism. Here, we have adopted the convention that
\begin{equation}
\HOM_A(M,N)=\bigoplus_{i\in \Z}\Hom_A^i(M,N),
\end{equation}
where $\Hom_A^i(M,N)$ stands for the space of homogeneous $A$-module maps of degree $i$.
Any $p$-DG module $M$ is quasi-isomorphic to some cofibrant module $\mathbf{p}(M)$ (e.g.~its bar resolution), and 
\begin{equation}\label{eqn-p-dg-morphism-space}
\Hom_{\mc{D}(A)}(M, N)\cong \Hom_{\mc{H}(A)}(\mathbf{p}(M),N)\cong \mH_{/0}^0(\HOM_A(\mathbf{p}(M), N)).
\end{equation}

\begin{rem}\label{rmk-enhancement}
At this point, we should emphasize that the triangulated category $\mc{D}(A)$ has an \emph{enhanced $p$-DG structure} on it, that is, the existence of an ``internal Hom'' enriched in $p$-complexes. Namely, for any cofibrant $p$-DG modules $M$, $N$, the $p$-complex $\HOM_A(M,N)$, up to homotopy, governs the morphism spaces in the derived category according to \eqref{eqn-p-dg-morphism-space}. In particular, given a finite collection of cofibrant $p$-DG modules $M_i$ ($i\in I$), it will usually be convenient to consider the $p$-DG endomorphism algebra
\begin{equation}
\mathrm{END}_A(\bigoplus_{i\in I}M_i)\cong \bigoplus_{i,j\in I}\HOM_A(M_i,M_j)
\end{equation}
which is equipped with the $p$-differential 
\[
\dif(f):=\dif_{M_j}\circ f-f\circ\dif_{M_i}
\]
for any $f\in \HOM_A(M_i,M_j)$. The endomorphism algebra then descends to a strictly associative algebra-object in the homotopy category of $p$-complexes, which controls the morphism spaces among $M_i$'s in the derived category by passing to $\mH_{/0}^0$ (equation \eqref{eqn-p-dg-morphism-space}). 

In particular, in the derived category $\mc{D}(A)$, if the collection consists of only the cofibrant object $A$ as a left regular module over the $p$-DG algebra $A$, then the endomorphism $p$-DG algebra
\[
\END_A(A)=A
\]
recovers the algebra $A$ with the natural differential, while
\[
\End_{\mc{D}(A)}(A)=\mH_{/0}^0(A).
\]
\end{rem}

A $p$-DG module $M$ is called \emph{compact} if its image in the derived category satisfies the condition that
\[
\mathrm{Hom}_{\mc{D}(A)}(M,-):\mc{D}(A)\lra \Bbbk-\mathrm{mod}
\]
commutes with arbitrary direct sums.
The \emph{compact derived category} $\mc{D}^c(A)$, the strictly full subcategory in $\mc{D}(A)$ consisting of compact objects,  is a categorification of its own \emph{Grothendieck group} $K_0(A):=K_0(\mc{D}^c(A))$. This Grothendieck group carries a module structure over the ring
\begin{equation}
\mathbb{O}_p:=K_0(\mc{D}^c(\Bbbk))\cong \frac{\Z[v^{\pm 1}]}{\Psi_p(v^2)}.
\end{equation}
Here $\Psi_p$ is the cyclotomic polynomial $\Psi_p(v)=1+v+\cdots+ v^{(p-1)}$, and $\mc{D}^c(\Bbbk)$ is the derived, or equivalently, homotopy category of finite-dimensional $p$-complexes over the ground field $\Bbbk$. Multiplication by $v$ comes from the decategorification of the grading shift functor on $p$-complexes. To avoid confusion, we will denote the image of $v$ under the canonical map $\Z[v^{\pm 1}]\lra \mathbb{O}_p$ by $q$, so that $q^{2p}=1$.

\paragraph{Slash-zero formality.} Below let us introduce a class of $p$-DG algebras that always satisfy the conditions of Lemma \ref{lemma-Kunneth}. The   slash cohomology of such an algebra is always an algebra with the trivial $p$-DG structure.

\begin{defn}\label{def-slash-zero-formal} 
A $p$-DG algebra $A$ is said to be \emph{slash-zero formal} if there is a quasi-isomorphism between $A$ and $\mH_{/0}(A)$.
\end{defn}

\begin{rem}[On the notion of quasi-isomorphism]\label{rmk-roofs}
We will allow a more general notion of quasi-isomorphic of $p$-DG algebras, as is usually adopted in the DG case. We will say that two $p$-DG algebras $A$, $B$ are \emph{quasi-isomorphic} if there is a sequence $p$-DG algebras $Z_1, Z_2,\dots, Z_k$  and  ``roofs'' of  quasi-isomorphisms connecting them:
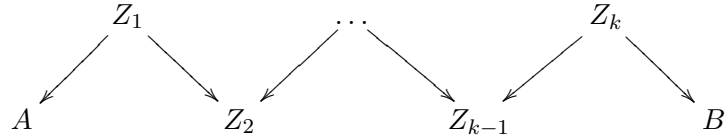
\[
\xymatrix{
 & Z_1\ar[dl]\ar[dr] && \dots \ar[dr]\ar[dl]  && Z_k \ar[dl]\ar[dr] & \\
A && Z_2  && Z_{k-1} && B 
}
\]
The resulting derived categories of $\mc{D}(A)$ and $\mc{D}(B)$ are equivalent under repeated applications of induction and restriction functors (see \cite[Corollary 8.18]{QYHopf}).

In particular, a $p$-DG algebra $A$ will be called \emph{slash-zero formal} if $A$ and $\mH_{/0}(A)$ are quasi-isomorphic. 
\end{rem}

\begin{example}\label{eg-easy-formal-p-dga}
We list some elementary examples of $p$-DG algebras that are slash-zero formal.
\begin{itemize}
\item[(i)] If $A$ is a usual $\Bbbk$-algebra with the trivial $p$-differential, then the $p$-DG algebra $(A,\dif_0\equiv 0)$ is slash-zero formal.
\item[(ii)] If $A$ is an acyclic $p$-DG algebra (i.e., $\mH_/(A)=0$), then the inclusion of $0$ into the algebra is a quasi-isomorphism and $A$ is thus slash-zero formal.
\item[(iii)] Consider the polynomial algebra $\Bbbk[x]$ with the differential $\dif(x^k)=kx^{k+1}$. It is readily seen that the inclusion of the unit map
\[
(\Bbbk, \dif_0)\lra (\Bbbk[x], \dif),\quad 1\mapsto 1
\]
is a $p$-DG algebra homomorphism that induces an isomorphism on slash cohomology. Thus $(\Bbbk[x],\dif)$ is slash-zero formal. The canonical projection map $\Bbbk[x]\lra \Bbbk$, $x\mapsto 0$ is another $p$-DG algebra homomorphism that realizes the quasi-isomorphism.
\end{itemize}
We will see more examples that are similar to (iii) in the next subsection.
\end{example}

\paragraph{Derived categories of slash-zero formal $p$-DG algebras.} When a $p$-DG algebra $A$ is slash-zero formal, the derived category $\mc{D}^c(A)$ can be identified with that of its slash cohomology algebra as follows.

\begin{lem}\label{lemma-formal-derived-equivalence}
Let $A$ be $p$-DG algebra together with a quasi-isomorphism $\phi: A \lra \mH_{/0}(A)$. Then derived induction and restriction along $\phi$ induces equivalences of triangulated categories
\[
\phi^*: \mc{D}^c(A)\lra \mc{D}^c(\mH_{/0}(A)), \quad \quad \phi_*:\mc{D}^c(\mH_{/0}(A))\lra \mc{D}^c(A)
\]
that are quasi-inverses of each other.
\end{lem}
\begin{proof}
This follows from the more general fact that, if $\phi:B\lra A$ is a quasi-ismorphism of $p$-DG algebras, then
\[
\phi^*: \mc{D}(A)\lra \mc{D}(B), \quad \quad \phi_*:\mc{D}(B)\lra \mc{D}(A)
\] 
are equivalences of triangulated categories that are quasi-inverse to each other. See \cite[Corollary 8.18]{QYHopf}.
\end{proof}

The Lemma implies that, when $A$ is slash-zero formal, its $p$-DG Grothendieck group $K_0(A)$ can be computed from the usual Grothendieck group of the graded associative algebra $\mH_/(A)$. Let us temporarily write the usual Grothendieck group of $\mH_/(A)$ as $K_0^\prime(\mH_/(A))$.  Below we denote the decategorification of the grading shift functor on $\mH_/(A)$-modules by $v$, so that $K_0^\prime(\mH_/(A))$ becomes a $\Z[v^{\pm 1}]$-module. We will only state the following less general version that we will use later.

\begin{cor}\label{cor-formal-pDGA-K-group}
Let $A$ be a slash-zero formal $p$-DG algebra whose $\mH_/(A)$ is a graded artinian algebra of finite homological dimension. Then its $p$-DG Grothendieck group equals
\[
K_0(A)\cong K_0^\prime(\mH_/(A))\otimes_{\Z[v^{\pm 1}]}\mathbb{O}_p.
\]
\end{cor}
\begin{proof}
Lemma \ref{lemma-formal-derived-equivalence} shows that we can compute $K_0(\mc{D}^c(A))$ from $K_0(\mc{D}^c(\mH_/(A)))$. When $\mH_/(A)$ is graded artinian and has finite homological dimension, the isomorphism
$$K_0(\mc{D}^c(\mH_/(A)))\cong K_0^\prime (\mH_/(A))\otimes_{\Z[v^{\pm 1}]}\mathbb{O}_p$$
follows from \cite[Proposition 9.10]{QYHopf}.
\end{proof}

\subsection{Symmetric polynomials and slash cohomology} In this Section, we recall the computation of the slash cohomology ring of $p$-DG symmetric polynomials done in \cite{EQ1,EQ2}, and we will deduce some useful consequences for later use. 

\paragraph{$p$-DG symmetric polynomials.} The ring of symmetric polynomials in $n$-variables $\sym_n$, which we identify as a subalgebra in the algebra of $n$-variable polynomials
$\pol_n:=\Bbbk[x_1,\dots, x_n]$, consists of polynomials invariant under the natural permutation action of $S_n$ on $\pol_n$. Each variable has degree $\mathrm{deg}(x_i)=2$.

Equip $\pol_n$ with the $p$-differential defined on the generators by $\dif(x_i)=x_i^2$, $i=1,\dots, n$, and extend the differential to $\pol_n$ by the Leibnitz rule. It is easy to see that, when $\mathrm{char}(\Bbbk)=p>0$, $(\pol_n,\dif)$ is a $p$-DG algebra. Furthermore, $\sym_n\subset \pol_n$ is preserved under the differential. One can also see this from the following differential action on the \emph{elementary symmetric polynomials} (see \cite[Lemma 3.1]{EQ1})
\begin{equation}\label{eqn-dif-elementary-function}
\dif(e_r)=e_1e_r-(r+1)e_{r+1} \ \ (1\leq r \leq n-1), \quad \quad \quad \dif(e_n)=e_1e_n,
\end{equation}
where, as usual, 
\[
e_k=e_k(x_1,\dots, x_n):= \sum_{1\leq i_1<\dots < i_k \leq n}x_{i_1}\cdots x_{i_k}.
\]
When talking about the \emph{$p$-DG algebra $\sym_n$}, we will always mean that $\sym_n$ is equipped with this differential \eqref{eqn-dif-elementary-function}. Likewise, one has the differential action on the \emph{$k$th complete symmmetric polynomials}
\[
h_k=h_k(x_1,\dots, x_n):= \sum_{1\leq i_1\leq \dots \leq i_k \leq n}x_{i_1}\cdots x_{i_k},
\]
for any $k\in \N$, given by
\begin{equation}
\label{eqn-dif-complete-function}
\dif(h_r)=-h_1h_r+(r+1)h_{r+1}.
\end{equation}

There is a one-parameter family of rank-one $p$-DG modules $\s_n(a)$, parametrized by $a\in \F_p$, over the $p$-DG algebra $(\sym_n,\dif)$. As $\sym_n$-modules, $\s_n(a)$ is freely generated by a degree-zero generator $\s_n(a):=\sym_n\cdot v_a$, with the differential acting on the module by
\begin{equation}\label{eqn-dif-rank-one-mod}
\dif(fv_a):=\dif(f)v_a+ae_1fv_a,
\end{equation}
where $f\in \sym_n$. When $a=0$, we also simplify the notation by $\s_n:=\s_n(0)$.

\paragraph{Slash-zero formality.} We recall the following result about the slash cohomology of $\sym_n$ and the rank-one module $\s_n(a)$.

\begin{lem}\label{lemma-slash-homology-rank-one-mod}
Let $n=kp+r$ be a natural number with $k\in \N$ and $0\leq r \leq p-1$. 
\begin{itemize}
\item[(i)]The $p$-DG algebra $\sym_n$ is slash-zero formal. Moreover, the natural inclusion map
\[
\Bbbk[e_p^p,\dots, e_{kp}^p]\lra \sym_n,
\]
where $\Bbbk[e_p^p,\dots, e_{kp}^p]$ has the zero differential, is a quasi-isomorphism of $p$-DG algebras. Likewise, the slash-zero formality is induced from the inclusion
\[
\Bbbk[h_p^p,\dots, h_{kp}^p]\lra \sym_n.
\]
\item[(ii)] When $a \in \{1,\dots, r\}$, the rank-one module $p$-DG module $\s_n(a)$ is acyclic: $\mH_/(\s_n(a))=0$.
\end{itemize}
\end{lem}
\begin{proof}
See \cite[Proposition 2.16]{EQ2}.
\end{proof}

We remark that the involution
\begin{equation}\label{eqn-sym-involution}
\omega:\sym_n\lra \sym_n,\quad \quad e_k\mapsto (-1)^kh_k,~(1\leq k \leq n)
\end{equation}
intertwines the differential defined on $\sym_n$ according to equation \eqref{eqn-dif-elementary-function} and \eqref{eqn-dif-complete-function}. The two inclusions of slash-zero cocycles are thus exchanged under the involution. For this reason, we will only formulate results concerning the slash cohomology of $\sym_n$ (or $\sym:=\lim_{n\to \infty}\sym_n$) in terms of the elementary symmetric polynomials (functions).

For the next result, let us embed $\sym_{(a+b)p}$ and $\sym_{ap}\otimes \sym_{bp}$ inside $\Bbbk[\underline{x}]\otimes \Bbbk[\underline{x^\prime}]$, where $\underline{x}$ stands for the set of variables $\{x_1,x_2,\dots, x_{ap}\}$ and likewise $\underline{x}^\prime$ stands for $\{x^\prime_1, x^\prime_2,\dots, x^\prime_{bp}\}$. By Lemma \ref{lemma-slash-homology-rank-one-mod}, we identify
\begin{equation}
\mH_{/}(\sym_{(a+b)p})\cong \Bbbk[e_p^p(\underline{x},\underline{x}^\prime),\dots, e_{(a+b)p}^p(\underline{x},\underline{x}^\prime)],
\end{equation}
and
\begin{equation}
\mH_{/}(\sym_{ap})\otimes \mH_{/}(\sym_{bp})\cong \Bbbk[e_p^p(\underline{x}),\dots, e_{kp}^p(\underline{x})]\otimes
\Bbbk[e_p^p(\underline{x}^\prime),\dots, e_{lp}^p(\underline{x}^\prime)],
\end{equation}

\begin{cor}\label{cor-slash-homology-inclusion}
Let $a,b\in \N$ be integers. The natural inclusion of $p$-DG algebras
\[
\iota: \sym_{(a+b)p} \lra \sym_{ap}\otimes \sym_{bp}
\]
induces an inclusion of slash cohomology $p$-DG algebras with zero differentials
\begin{eqnarray*}
\iota_/:  \Bbbk[e_p^p(\underline{x},\underline{x}^\prime),\dots, e_{(a+b)p}^p(\underline{x},\underline{x}^\prime)] & \lra & \Bbbk[e_p^p(\underline{x}),\dots, e_{ap}^p(\underline{x})]\otimes
\Bbbk[e_p^p(\underline{x}^\prime),\dots, e_{bp}^p(\underline{x}^\prime)]\\
e_{rp}^p(\underline{x},\underline{x}^\prime) & \mapsto & \sum_{i=0}^{\mathrm{min}(r,a,b)} e_{ip}^p (\underline{x}) \otimes e_{(r-i)p}^p (\underline{x}^\prime) ,
\end{eqnarray*}
where $r\in \{1,\dots, a+b\}$. In particular, the induced map on slash cohomology is an embedding of polynomial algebras.
\end{cor}
\begin{proof}
Elementary symmetric functions, when the variables are split into two independent sets $\underline{x}$ and $\underline{x}^\prime$, satisfy
\[
e_i(\underline{x},\underline{x}^\prime)=\sum_{j=0}^{\mathrm{min}(i,ap,bp)}e_{j}(\underline{x})\cdot e_{i-j}(\underline{x}^\prime).
\]
Taking $i=rp$ and raising the above equation to the $p$th power, we obtain
\[
e_{rp}^p(\underline{x},\underline{x}^\prime)=\sum_{j=0}^{\mathrm{min}(r,a,b)}e_{j}^p(\underline{x})\cdot e^p_{rp-j}(\underline{x}^\prime),
\]
where we have used that, in characteristic $p>0$, the $p$th power map (Frobenius homomorphism) preserves the abelian group structure of $\Bbbk$-algebras. If $p$ does not divide $j$, Lemma \ref{lemma-slash-homology-rank-one-mod} shows that the $0$-cocycle $e_{j}^p(\underline{x})$ is equal to zero in the slash cohomology. The result thus follows.
\end{proof}
 
\begin{rem}\label{rmk-undivisible-splitting}
Notice that, if we split the set of $(k+l)p$ variables into two groups such that neither is divisible by $p$, then the induced map on slash cohomology will not be injective. For instance, take $a=b=1$, and consider the inclusion
\[
\iota:\sym_{2p}(\underline{x},\underline{x}^\prime)\subset \sym_{p-1}(\underline{x})\otimes \sym_{p+1}(\underline{x}^\prime).
\] 
Using Lemma \ref{lemma-slash-homology-rank-one-mod} again, we see that, on the level of slash cohomology, the induced map
\[
\iota_{/}: \Bbbk[e_p^p(\underline{x},\underline{x}^\prime),e_{2p}^p(\underline{x},\underline{x}^\prime)]\lra \Bbbk\otimes \Bbbk[e_p^p(\underline{x}^\prime)]
\]
sends $e_{2p}^{p}(\underline{x},\underline{x}^\prime)$ to zero on the right hand side, and the result of Corollary \ref{cor-slash-homology-inclusion} fails in this situation.
\end{rem}

Next, for each $1\leq i \leq k$, we abbreviate the set of variables $\underline{x}^{(i)}:=\{x_{(i-1)p+1},\dots, x_{ip}\}$. Then we have an inclusion of $p$-DG algebras
\[
\iota: \sym_{kp}(\underline{x}^{(1)},\dots, \underline{x}^{(k)})\subset \sym_{p}(\underline{x}^{(1)})\otimes\cdots\otimes \sym_p(\underline{x}^{(k)}).
\]
Passing to slash cohomology, we obtain
\begin{equation}
\iota_/:\Bbbk[e_p^p(\underline{x}^{(1)},\dots, \underline{x}^{(k)}),\dots, e_{kp}^p(\underline{x}^{(1)},\dots, \underline{x}^{(k)})]\lra \Bbbk[e_p^p(\underline{x}^{(1)})]\otimes\cdots \otimes \Bbbk[e_p^p(\underline{x}^{(k)})].
\end{equation}
It will be convenient to define the auxiliary variables
\begin{equation}\label{eqn-y-variables}
y_i:=e_{p}^p(\underline{x}^{(i)})=e_p^p(x_{(i-1)p+1},\dots, x_{ip}).
\end{equation}
Clearly $\mathrm{deg}(y_i)=2p^2$.

\begin{cor}\label{cor-slash-cohomology-inclusion-2}
On the level of slash cohomology, the inclusion of $p$-DG algebras
$$\iota: \sym_{kp} \lra (\sym_{p})^{\otimes k}$$
induces an isomorphism of $\mH_/(\sym_{kp})$ onto the space of symmetric polynomials in $y_i$ ($i=1,\dots, k$). 
\end{cor}
\begin{proof}
The proof is similar to that of Corollary \ref{cor-slash-homology-inclusion}. Indeed, for any $e_{rp}^p\in \sym_{kp}$, $1\leq r\leq k$, we have
\[
e_{rp}^p(\underline{x}^{(1)},\dots, \underline{x}^{(k)}) = \sum_{j_1+\dots + j_k=rp}e_{j_1}^p(\underline{x}^{(1)})\cdots e_{j_k}^p(\underline{x}^{(k)}).
\]
Using Lemma \ref{lemma-slash-homology-rank-one-mod}, each $e_{j_a}^p(\underline{x}^{(a)})=0$, $a=1,\dots, k$, in the slash cohomology ring of $\sym_{p}(\underline{x}^{(a)})$ unless $p$ divides $j_a$. Since $0\leq j_a\leq p$, the remaining terms are equal to  $e_r(y_1,\dots, y_k)$. The claim follows.
\end{proof}

\begin{rem}\label{rmk-infinite-version}By passing to the $n\to\infty$ version of Lemma \ref{lemma-slash-homology-rank-one-mod} (see \cite[Proposition 3.8]{EQ1}), one sees that, the slash cohomology is isomorphic to
\[
\mH_/(\sym)\cong \Bbbk[e_p^p,e_{2p}^p,e_{3p}^p,\dots] \cong \Bbbk[h_p^p,h_{2p}^p,h_{3p}^p,\dots],
\]
and the coproduct on the Hopf algebra $\sym$ induces an coassociative coproduct on $\mH_/(\sym)$:
\[
\Delta_/:\mH_/(\sym)\lra \mH_/(\sym)\otimes \mH_/(\sym),\quad
e_{kp}^p\mapsto \sum_{a+b=k}e_{ap}^p\otimes e_{bp}^p,\quad \quad
h_{kp}^p\mapsto \sum_{a+b=k}h_{ap}^p\otimes h_{bp}^p.
\]
The antipode $\omega:\sym\lra \sym$, $e_i\mapsto (-1)^{i}h_i$ gives rise to
\[
\omega_/:\mH_/(\sym) \lra \mH_/(\sym),\quad \quad e_{kp}^p\mapsto (-1)^{kp^2}h_{kp}^p=(-1)^{k}h_{kp}^p.
\]
In this way, the slash cohomology ring $\mH_/(\sym)$ is just isomorphic to another copy of symmetric functions $\sym^\prime\cong \Bbbk[e_1^\prime,e_2^\prime,e_3^\prime,\dots]$ (with the zero differential), but the degree is adjusted to be $\mathrm{deg}(e_i^\prime)=2ip^2$. We have a homogeneous ``thickening map'' $\Theta_0$:
\begin{equation}
\Theta_0:\sym^\prime \lra \mH_/(\sym),\quad \quad
\Theta_0(e_i^\prime):=e_{ip}^p,
\end{equation}
which is a quasi-isomorphism of $p$-DG Hopf algebras.
\end{rem}

\section{Grassmannian bimodules and unoriented thick calculus}\label{sec-Gr-bim}
The $p$-DG (thick) nilHecke algebra arises as the $p$-DG endomorphism algebra of some Grassmannian bimodules \cite[Section 3.2]{EQ2}, and categorifies the nilpotent part of quantum $\mf{sl}_2$ generated by all divided powers of ${E}$ at a prime root of unity. In this Section, we recall the construction, and then we investigate how ``one-half'' of the quantum Frobenius map is categorified by passing to the full subcategory monoidally generated by $E^{(p)}$.

\subsection{A quantum Frobenius map}\label{sec-half-sl2}
Following Lusztig \cite{Lus4}, we will define the upper nilpotent part of the quantum $\mathfrak{sl}_2$ at a generic value $v$, which is denoted for short by $U_v^+$, as the $\Z[v^{\pm 1}]$-algebra spanned by
\begin{equation}\label{eqn-half-sl2-generic}
U^+_v:=\bigoplus_{a\in \N}\Z[v^{\pm 1}]\cdot \theta^{(a)}.
\end{equation}
with the multiplication and comultiplication given by
\begin{equation}\label{eqn-half-sl2-generic-multiplication}
m_v: U^+_{v}\otimes_{\Z[v^{\pm 1}]}U^+_{v}\lra  U^+_{v}, \quad m_v(\theta^{(a)}\theta^{(b)})={a+b \brack a}_v \theta^{(a+b)},
\end{equation}
\begin{equation}\label{eqn-half-sl2-generic-comultiplication}
r_v:  U^+_{v}\lra  U^+_{v}\otimes_{\Z[v^{\pm 1}]}U^+_{v}, \quad r_v(\theta^{(a)})=\sum_{k=0}^{a}v^{k(k-a)}\theta^{(k)}\otimes \theta^{(a-k)}.
\end{equation}
Here we have used the quantum integer $[n]_v:=\frac{v^{n}-v^{-n}}{v-v^{-1}}=\sum_{i=0}^{n-1} v^{n -1-2i}$, and the quantum binomial coefficient
\[
{a+b \brack a}_v:= \frac{[a+b]_v!}{[a]_v![b]_v!}.
\]

Next we recall the quantum group $U^+_{\mathbb{O}_p}$ at a prime root of unity as a twisted bialgebra and its quantum Frobenius map. This is the twisted bialgebra obtained from 
$U^+_v$ by the canonical base change map $\Z[v^{\pm 1}]\lra \mathbb{O}_p$, $v\mapsto q$. More specifically, as an algebra, $U^+_{\mathbb{O}_p}$ is a free $\mathbb{O}_p$-module spanned by $E^{(a)}$, $a\in \N$,
\[
U^+_{\mathbb{O}_P}=\bigoplus_{a\in \N} \mathbb{O}_p\cdot E^{(a)},
\]
equipped with the multiplication 
\begin{equation}
m: U^+_{\mathbb{O}_p}\otimes_{\mathbb{O}_p}U^+_{\mathbb{O}_p}\lra  U^+_{\mathbb{O}_p}, \quad m(E^{(a)}E^{(b)})={a+b \brack a}_{\mathbb{O}_p} E^{(a+b)},
\end{equation}
and the comultiplication 
\begin{equation}
r:  U^+_{\mathbb{O}_p}\lra  U^+_{\mathbb{O}_p}\otimes_{\mathbb{O}_p}U^+_{\mathbb{O}_p}, \quad r(E^{(a)})=\sum_{k=0}^{a}q^{k(k-a)}E^{(k)}\otimes E^{(a-k)}.
\end{equation}
Here the coefficients ${a+b \brack a}_{\mathbb{O}_p}$ and $q^{k(k-a)}$ are evaluated in the ground ring $\mathbb{O}_p$.

Define another base change ring homomorphism
\begin{equation}\label{eqn-base-change-map}
\rho:\Z[v^{\pm 1}]\lra \mathbb{O}_p, \quad v\mapsto q^{p^2}=q^p,
\end{equation} 
if $p$ is odd, where we have used that $q^{2p}=1\in \mathbb{O}_p$. When $p=2$, we set $\rho(v)=1$ if $p=2$. 

Under $\rho$, we have, when $p$ is odd,
\begin{equation}
\rho([n]_v)=\rho\left(\frac{v^n-v^{-n}}{v-v^{-1}}\right) = q^{(1-n)p}(1+q^{2p}+\cdots+q^{2p(n-1)})=nq^{(1-n)p},
\end{equation}
 for any $n\in \N$. It follows that 
\begin{equation}\label{eqn-evaluation-binomial-coeff}
\rho({a+b\brack a}_v)=q^{pab}{a+b\choose a}.
\end{equation}
Likewise, when $p=2$, we have $\rho ([n]_v)=n$ and $\rho ({a+b\brack a}_v) = {a+b\choose a}$.

\begin{rem}\label{rem-wrong-grading-choice}
Our (unfortunate) grading choice for the differential to be of degree two was done in order to match the usual representation theoretical conventions. This has the effect that, when $p$ is odd, the ring $\mathbb{O}_p$ is no longer an integral domain. Moreover, in $\mathbb{O}_p$, we have $q^{2p}=1$ and $q^p\neq -1$. On the other hand, there are two canonical ring homomorphisms
\[
\mathbb{\varrho}_1: \mathbb{O}_p\lra \mathcal{O}_{2p}=\Z[v]/(\Psi_{2p}(v)), \quad \quad
\mathbb{\varrho}_2: \mathbb{O}_p\lra \mathcal{O}_{p}=\Z[v]/(\Psi_{p}(v)),
\]
because $1+v^2+\cdots + v^{2(p-1)}$ factors into the product of the $p$th and $2p$th cyclotomic polynomials $\Psi_{p}(v)$ and $\Psi_{2p}(v)$.
If we perform a further base change along $\varrho_1$, then $q^p=-1$ in the cyclotomic ring $\mathcal{O}_{2p}$, while doing so along $\varrho_2$ makes $q^p=1$.
\end{rem}

When $U^+_v$ is base changed along $\rho$, we see from equations \eqref{eqn-half-sl2-generic},\eqref{eqn-half-sl2-generic-multiplication} and \eqref{eqn-half-sl2-generic-comultiplication} that 
\begin{equation}
U^+_\rho:= U^+_v\otimes_{\Z[v^{\pm 1}],\rho} \mathbb{O}_p
\end{equation}
 is closely related to an $\mathbb{O}_p$-integral form of the classical universal enveloping algebra $U^+(\mathfrak{sl}_2)$. We will denote the generator $\theta^{(a)}\otimes_{\rho} 1$ by $\mathsf{E}^{(a)}$, so that
\begin{equation}
U^+_{\rho}\cong 
\bigoplus_{a\in \N}\mathbb{O}_p \cdot \mathsf{E}^{(a)},
\end{equation}
with the multiplication structure
\begin{equation}
\mathsf{E}^{(a)}\mathsf{E^{(b)}}=
\left\{
\begin{array}{cc}
q^{pab}{a+b \choose a} \mathsf{E}^{(a+b)} & p\neq 2,\\
& \\
{a+b \choose a} \mathsf{E}^{(a+b)} & p=2.
\end{array}
\right.
\end{equation}
This equation tells us that, if we use $\varrho_2$ instead of $\rho$ as in Remark \ref{rem-wrong-grading-choice}, then $q^p=1$, and we recover the usual universal enveloping algebra $U^+(\mf{sl}_2)$ over $\mathcal{O}_p$ with divided power generators. 

We now recall the quantum Frobenius map for the upper nilpotent part of $\mf{sl}_2$ following \cite[Chapter 35]{Lus4}, which is adapted to our setting over $\mathbb{O}_p$.

\begin{defn}\label{def-quantum-Frob-half} The \emph{quantum Frobenius map} is the $\mathbb{O}_p$-algebra homomorphism
\[
\mathrm{Fr}: U^+_{\mathbb{O}_p} \lra U^{+}_\rho, \quad E^{(a)}\mapsto
\left\{
\begin{array}{ll}
\mathsf{E}^{(a/p)} & \mathrm{if}~p~|~a,\\
0 & \mathrm{if}~p \nmid a.
\end{array}
\right.
\]
\end{defn}

\subsection{Grassmannian bimodules}
Let $(a,b)\in \N^2$ be a pair of natural numbers, and consider the $GL_{a+b}$-equivariant cohomology ring of $\mathrm{Gr}(a,a+b)$:
\[
\mH^*_{GL(a+b)}(\mathrm{Gr}(a,a+b),\Bbbk)\cong \sym_a\otimes_\Bbbk \sym_b.
\]  
We will usually just write the right hand side as $\sym_{a,b}$ for short. Each $\sym_a$, $\sym_b$ will be equipped with the $p$-DG structure defined by \eqref{eqn-dif-elementary-function}. When necessary, we will identify elements in $\sym_a$ as symmetric polynomials in the set of variables $\underline{x}=\{x_1,\dots, x_a\}$ while $\sym_b$ as in $\underline{x}^\prime=\{x_1^\prime,\dots, x_b^\prime\}$. Under the differential, $\sym_a\otimes \sym_b$ is equipped with the tensor product $p$-DG algebra structure: for any $\pi_1(\underline{x})\in \sym_a(\underline{x})$ and $\pi_2(\underline{x}^\prime)\in \sym_b(\underline{x}^\prime)$, we have 
\begin{equation}
\dif(\pi_1(\underline{x})\otimes \pi_2(\underline{x}^\prime))=\dif(\pi_1(\underline{x}))\otimes\pi_2(\underline{x}^\prime)+\pi_1(\underline{x})\otimes \dif(\pi_2(\underline{x}^\prime)).
\end{equation}
Furthermore, $\sym_{a,b}$ contains $\sym_{a+b}$ canonically as a $p$-DG subalgebra.

\begin{defn}\label{def-Grassmannian-module}
Let $(a,b)\in \N^2$ be a pair of natural numbers. The rank-one $p$-DG module 
$$\s_{a,b}=\sym_{a,b}\cdot v_{a,b}$$
is equipped with the differential action on the generator given by
\[
\dif(v_{a,b})=-a e_1(\underline{x}^\prime)v_{a,b}.
\]
The degree\footnote{
The degree of the generator is only essential for Lemma \ref{lem-dif-basis-Grassmann-module} and Theorem \ref{thm-weak-categorification-half-sl2}, where it is needed to make some quantum binomial coefficients in $v$ symmetric with respect to $v\mapsto v^{-1}$.} of the module generator $v_{a,b}$ is fixed to be $-ab$. More generally, given any sequence of natural numbers $\underline{a}:=(a_1,\dots, a_r)\in \N^{r}$, define
\[
\s_{a_1, a_2,\dots, a_r}:=(\sym_{a_1}\otimes \sym_{a_2}\otimes\cdots \otimes \sym_{a_{r}})\cdot v_{\underline{a}}
\]
with the differential action determined on the generator by
\[
\dif(v_{\underline{a}}):=\sum_{i=2}^r -(a_1+\dots+ a_{i-1})e_1(x_{a_1+\dots +a_{i-1}+1},\dots, x_{a_1+\dots +a_{i-1}+a_i})v_{\underline{a}},
\]
and extended to the entire module via the Leibniz rule. The degree of $v_{\underline{a}}$ is taken to be 
$$\mathrm{deg}(v_{\underline{a}}):=-\sum_{i\neq j}a_ia_j.$$
\end{defn}

Let us also fix some standard combinatorial notation. For a pair of natural numbers $(a,b)$, denote by $P(a,b)$ the set of Young diagrams/partitions that fit into an $a\times b$ rectangle. When $b\to \infty$, we will just abbreviate $P(a):=P(a,\infty)$. A partition $\lambda=(\lambda_1\geq \lambda_2\geq \dots \geq \lambda_a\geq 0)\in P(a)$ gives rise to a symmetric polynomial in $a$ variables known as the \emph{Schur polynomial} $\pi_{\lambda}(\underline{x})$. We recall the following useful differential action of $\dif$ on $\pi_\lambda\in \sym_a$, which generalizes the actions \eqref{eqn-dif-elementary-function} and \eqref{eqn-dif-complete-function}:
\begin{equation}\label{eqn-dif-on-Schur}
\dif(\pi_\lambda) =\sum_{\mu=\lambda+\square}\mathrm{C}(\square)\pi_{\mu}.
\end{equation}
Here $\mu=\lambda+\square$ stands for a Young diagram in $P(a)$ that is obtained from $\lambda$ by adding a single box, and $\mathrm{C}(\square)$ stands for the content number of the box added (i.e., the column number of the box minus its row number). See \cite[Lemma 2.4]{EQ2} for the proof of \eqref{eqn-dif-on-Schur}.

\begin{lem}\label{lem-dif-basis-Grassmann-module}
\begin{enumerate}
\item[(i)] Considered as a left $p$-DG module over $\sym_{a+b}$, $\s_{a,b}$ is compact cofibrant of graded rank ${a+b\brack b}_v$. A $\dif$-stable $\sym_{a+b}$-basis for the module is given by
\[
\left\{
(1\otimes \pi_{\lambda}(\underline{x}^\prime))v_{a,b}|\lambda \in P(b,a)
\right\}.
\]
\item[(ii)] The dual module $\s_{a,b}^\vee:=\HOM_{\sym_{a+b}}(\s_{a,b},\sym_{a+b})$ is also compact cofibrant with the $\dif$-stable dual basis
\[
\left\{
( \pi_{\lambda}(\underline{x})\otimes 1)v_{a,b}^\vee|\lambda \in P(a,b)
\right\}.
\]
The induced differential on the dual module generator is given by
\[
\dif(v_{a,b}^\vee)=-be_1(\underline{x})v_{a,b}^\vee.
\]
\end{enumerate}
\end{lem}
\begin{proof}
This is \cite[Proposition 3.3]{EQ2}.
\end{proof}

\begin{defn}\label{def-p-DG-sym} The derived $p$-DG symmetric polynomials is the category
\[
\mc{D}(\sym):=\bigoplus_{a\in \N}\mc{D}(\sym_a),
\] 
equipped with the following  \emph{multiplication functor} $\mc{M}$:
\[
\mc{M}:\mc{D}(\sym_a)\times \mc{D}(\sym_b)\lra \mc{D}(\sym_{a+b}), \quad (M_1,M_2)\mapsto \s_{a,b}\otimes_{\sym_{a,b}}^{\mathbf{L}}(M_1\boxtimes M_2),
\]
and the \emph{comultiplication functor} $\mc{R}$:
\[
\mc{R}:\mc{D}(\sym_{a+b}) \lra \mc{D}(\sym_a)\times \mc{D}(\sym_b), \quad M \mapsto (\sym_a\otimes \sym_b)u_{a,b}\otimes_{\sym_{a+b}}^{\mathbf{L}}M,
\]
where the module generator $u_{a,b}$ has degree $-ab$, and the differential acts on the generator by $\dif(u_{a,b})=0$.
\end{defn}

In \cite[Theorem 3.15]{EQ2}, we have established the following.

\begin{prop}\label{prop-EQII-half-thm}
The derived category of $p$-DG symmetric polynomials categorifies $U_{\mathbb{O}_p}^+$ of quantum $\mathfrak{sl}_2$ at a $p$th root of unity as a twisted bialgebra:
\[
K_0(\mc{D}^c(\sym))\cong U^+_{\mathbb{O}_p}.
\]
Under this isomorphism, the symbol of the rank-one free $p$-DG module $[\s_a]$ is identified with the element $E^{(a)}$. The multiplication and comultiplication functor categorify respectively the multiplication and comultiplication of $U^+_{\mathbb{O}_p}$.
\end{prop}

In order to categorify the quantum Frobenius map for $U^+_{\mathbb{O}_p}$, we will next investigate a full subcategory of $\mc{D}(\sym)$ with its $p$-DG enhanced structure (see Remark \ref{rmk-enhancement}).

\begin{defn}
\label{def-Frobnenius-subcat} Let $\mc{D}_{(p)}(\sym)$ be the full subcategory inside $\mc{D}(\sym)$ which consists of
\[
\mc{D}_{(p)}(\sym):=\bigoplus_{a\in \N}\mc{D}(\sym_{ap}).
\]
The canonical embedding functor will be denoted 
\[
\jmath:\mc{D}_{(p)}(\sym)\lra \mc{D}(\sym).
\]
As a full subcategory of $\mc{D}(\sym)$,  $\mc{D}_{(p)}(\sym)$  inherites a multiplication and comultiplication functor from $\mc{M}$ and $\mc{R}$, which will be written respectively as $\mc{M}_{(p)}$ and $\mc{R}_{(p)}$.
\end{defn}

Componentwise, the functor $\mc{M}_{(p)}$ has the effect
\begin{equation}
\mc{M}_{(p)}:\mc{D}(\sym_{ap})\times \mc{D}(\sym_{bp})\lra \mc{D}(\sym_{(a+b)p}),\quad (M_1,M_2)\mapsto \s_{ap,bp}\otimes_{\sym_{ap,bp}}(M_1\boxtimes M_2).
\end{equation}
Likewise, the comultiplication functor $\mc{R}_{(p)}$ is seen to be
\begin{equation}
\mc{R}_{(p)}:\mc{D}(\sym_{(a+b)p})\lra \mc{D}(\sym_{ap})\times \mc{D}(\sym_{bp}),\quad M \mapsto \sym_{ap,bp}\otimes_{\sym_{(a+b)p}}^{\mathbf{L}}M.
\end{equation}

Now we are ready to state and prove a (weak) categorification result about the quantum Frobnenius map (Definition \ref{def-quantum-Frob-half}).

\begin{thm}\label{thm-weak-categorification-half-sl2}
The compact $p$-DG derived category $\mc{D}^c_{(p)}(\sym)$ categorifies the $\mathbb{O}_p$-integral universal enveloping algebra $U^+_\rho$:
\[
K_0(\mc{D}^c_{(p)}(\sym))\cong U^+_\rho.
\]
The multiplication functor $\mc{M}_{(p)}$ and comultiplication functor categorify the multiplication and comultiplication on $U^+_\rho$. Consequently, the restriction functor along $\jmath:\mc{D}^c_{(p)}(\sym)\lra \mc{D}^c(\sym)$ categorifies the quantum Frobenius map.
\end{thm}
\begin{proof}
By construction, we have
\[
K_0(\mc{D}^c_{(p)}(\sym))=\bigoplus_{a\in \N}K_0(\mc{D}^c(\sym_{ap})).
\]
Via Proposition \ref{prop-EQII-half-thm}, we need to analyze the $p$-DG algebras $\sym_{ap}$ ($a\in \N$) and the effect of the multiplication functor $\mc{M}_{(p)}$ on a pair of the rank-one free module $\s_{ap}$.

By Lemma \ref{lemma-slash-homology-rank-one-mod}, the $p$-DG algebra $\sym_{ap}$ is slash-zero formal, and there is a quasi-isomorphism induced from the inclusion of its slash cohomology ring
\[
\iota_{a}: \Bbbk[e_p^p,\dots, e_{ap}^p]\lra \sym_{ap}.
\] 
Notice that $\mathrm{deg}(e_{kp}^p)=k\mathrm{deg}(e_p^p)=2kp^2$ $(1\leq k \leq a)$. Thanks to Lemma \ref{lemma-formal-derived-equivalence}, we have an equivalence of derived categories by induction along $\iota_a$
\[
\iota^*_a: \mc{D}(\Bbbk[e_p^p,\dots, e_{ap}^p])\lra \mc{D}(\sym_{ap}).
\]
Then Corollary \ref{cor-formal-pDGA-K-group} shows that
\[
K_0( \mc{D}^c(\sym_{ap}))\cong K_0(\mc{D}^c(\Bbbk[e_p^p,\dots, e_{ap}^p]))\cong  \mathbb{O}_p.
\]

 To see that the multiplication functor has the desired effect, we compute that
 \[
\mc{M}_{(p)}(\s_{ap},\s_{bp}) \cong \s_{ap,bp}\otimes_{\sym_{ap,bp}}(\s_{ap}\boxtimes \s_{bp})\cong \s_{ap,bp}
 \]
 as a left $p$-DG module over $\sym_{(a+b)p}$. It follows that, by Lemma \ref{lem-dif-basis-Grassmann-module}, the following equality of symbols hold on the level of Grothendieck groups
 \[
m([\s_a][\s_b])=[\s_{ap,bp}]={(a+b)p \brack ap}_{\mathbb{O}_p}[\sym_{(a+b)p}] = q^{-pab}{a+b\choose a}[\s_{(a+b)p}].
 \]
In the last equality, we have used the standard binomial equality that 
\[
{(a+b)p \brack ap}_{\mathbb{O}_p} = q^{pab}{a+b\choose a} =\rho ({a+b \brack a}_v)
\]
 whose proof can be found, for instance, in \cite[Lemma 34.1.2]{Lus4}. Alternatively we refer the reader to Remark \ref{rmk-categorifying-binomial-reduction} later for a categorical ``explanation'' of this equality.
 
 The comultiplication functor $\mc{R}_{(p)}$ is dealt with similarly. We will leave this case to the reader as an exercise.
\end{proof}

\begin{rem}
Theorem \ref{thm-weak-categorification-half-sl2} also tells us that the induction functor along $\jmath$ gives rise to categorification of a section of the quantum Frobenius map $\mathrm{Fr}: U^+_{\mathbb{O}_p}\lra U^+_\rho$.
\end{rem}

\subsection{Unoriented graphical calculus}
Our goal in this section is to have a combinatorial study of the $2$-morphism spaces controlling the functor $\mc{M}_{(p)}$ in the spirit of Khovanov-Lauda \cite{KL1,Lau1} and Rouquier \cite{Rou2}. This will strengthen Theorem \ref{thm-weak-categorification-half-sl2} into a strong categorification result.

To start, we will consider the $p$-DG module $\s_{ap,bp}$ (Definition \ref{def-Grassmannian-module}) regarded as a left module over $\sym_{(a+b)p}$ and right module over $\sym_{ap,bp}$. Diagrammatically, we will depict elements of the module as
\[
\s_{ap,bp}\cong \left\{
\begin{DGCpicture}[scale=0.8]
\DGCPLstrand[thk1](1,0)(1,-1)[$_{(a+b)p}$]
\DGCPLstrand[thk1](1,-1)(0,-2)[`$_{ap}$]
\DGCPLstrand[thk1](1,-1)(2,-2)[`$_{bp}$]
\DGCcoupon(0.2,-1.25)(0.8,-1.75){$\scriptsize{x}$}
\DGCcoupon(1.2,-1.25)(1.8,-1.75){$\scriptsize{y}$}
\end{DGCpicture}\Bigg| x\in \sym_{ap}, y\in \sym_{bp}\right\},
\]
The labels $ap$, $bp$ and $(a+b)p$ indicate the thickness of the strands involved. When a strand has thickness $r$, it can carry coupons labeled by symmetric polynomials $x\in \sym_{r}$.  
\[
\begin{DGCpicture}
\DGCPLstrand[thk1](1,0)(1,-1)[`$^r$]
\DGCcoupon(0.625,-0.25)(1.375,-0.75){$\scriptsize{x}$}
\end{DGCpicture}
\]
The diagrams are
subject to the \emph{Grassmannian sliding relation}: for any $k\in \{1,\dots , (a+b)p\}$, we have
\[
\begin{DGCpicture}
\DGCPLstrand[thk1](1,0)(1,-1)[$_{(a+b)p}$]
\DGCPLstrand[thk1](1,-1)(0,-2)[`$^{ap}$]
\DGCPLstrand[thk1](1,-1)(2,-2)[`$^{bp}$]
\DGCcoupon(0.625,-0.25)(1.375,-0.75){$\scriptsize{e_k}$}
\end{DGCpicture}
=\sum_{l=0}^k
\begin{DGCpicture}
\DGCPLstrand[thk1](1,0)(1,-1)[$_{(a+b)p}$]
\DGCPLstrand[thk1](1,-1)(0,-2)[`$^{ap}$]
\DGCPLstrand[thk1](1,-1)(2,-2)[`$^{bp}$]
\DGCcoupon(0.125,-1.25)(0.875,-1.75){$\scriptsize{e_l}$}
\DGCcoupon(1.125,-1.25)(1.875,-1.75){$\scriptsize{e_{k-l}}$}
\end{DGCpicture}~,
\]
where it is understood that $e_m(x_1,\dots,x_n)=0$ if $m> n$.
The differential on the lowest degree generator is given by the formula (c.f.~Definition \ref{def-Grassmannian-module}):
\begin{equation}\label{eqn-d-action-dual-mod-generator}
\dif\left(~
\begin{DGCpicture}[scale=0.8]
\DGCPLstrand[thk1](0,0)(1,1)[$^{ap}$]
\DGCPLstrand[thk1](2,0)(1,1)[$^{bp}$]
\DGCPLstrand[thk1](1,1)(1,2)[`$_{(a+b)p}$]
\end{DGCpicture}
~\right)=
0.
\end{equation}
One can readily see that the differential is inherited from the one in \cite{EQ2}, as $ap,bp$ are zero modulo $p$.

The enhanced endomorphism algebra $\END_{\sym_{(a+b)p}}(\s_{ap,bp})$ is a size-${(a+b)p \choose ap}^2$ matrix $p$-DG algebra with coefficients in $\sym_{(a+b)p}$ because of Lemma \ref{lem-dif-basis-Grassmann-module}. It has the following diagrammatic $\Bbbk$-basis.
\begin{equation}\label{eqn-rank-one-mod-end-alg-diag}
\left\{
(-1)^{|\hat{\mu}|}
\begin{DGCpicture}[scale=0.8]
\DGCPLstrand[thk1](0,0)(1,1.15)[$^{ap}$]
\DGCPLstrand[thk1](2,0)(1,1.15)[$^{bp}$]
\DGCPLstrand[thk1](1,1.15)(1,1.85)
\DGCPLstrand[thk1](1,1.85)(0,3)[`$_{ap}$]
\DGCPLstrand[thk1](1,1.85)(2,3)[`$_{bp}$]
\DGCcoupon(1.15,0.35)(1.85,0.85){$\scriptsize{\pi_{\hat{\mu}}}$}
\DGCcoupon(0.15,2.15)(0.85,2.65){$\scriptsize{\pi_\lambda}$}
\DGCcoupon(0.65,1.3)(1.35,1.7){$\scriptsize{\pi_\nu}$}
\end{DGCpicture}~\Bigg|
\lambda, \mu \in P(ap,bp), \nu \in P((a+b)p)
\right\}.
\end{equation}
Here we have adopted the notation that, for any $c,d\in \N$ and any partition 
$$\mu=(\mu_1\geq \dots \geq \mu_{c-1} \geq \mu_c\geq 0)\in P(c,d),$$ 
its complement partition $\hat{\mu}\in P(d,c)$ is obtained from $\mu$ by the rule
\begin{equation}
\hat{\mu}^t:=(d-\mu_c,d-\mu_{c-1},\dots, d-\mu_1),
\end{equation}
where $t$ is the usual transpose on Young diagrams by reflection about the main diagonal.

The composition of two elements in this matrix presentation is given by vertical juxtaposition of diagrams. Whenever two diagrams are vertically composed, it is simplified according to the following ``pairing'' simplification rule:
\begin{equation}\label{eqn-composition-basis-elts}
\begin{DGCpicture}[scale=0.8]
\DGCPLstrand[thk1](0,0)(1,1.15)
\DGCPLstrand[thk1](2,0)(1,1.15)
\DGCPLstrand[thk1](1,1.15)(1,1.85)
\DGCPLstrand[thk1](0,0)(1,-1.15)
\DGCPLstrand[thk1](2,0)(1,-1.15)
\DGCPLstrand[thk1](1,-1.15)(1,-1.85)[`$^{(a+b)p}$]
\DGCcoupon(-0.35,-0.25)(0.35,0.25){$\scriptsize{\pi_\lambda}$}
\DGCcoupon(1.65,-0.25)(2.35,0.25){$\scriptsize{\pi_{\hat{\mu}}}$}
\end{DGCpicture}
= (-1)^{|\hat{\mu}|}~
\begin{DGCpicture}[scale=0.8]
\DGCstrand[thk1](0,0)(0,3.7)[$^{(a+b)p}$`]
\end{DGCpicture}
\ .
\end{equation}
The induced differential acts trivially on the lowest degree element of $\END_{\sym_{(a+b)p}}(\s_{ap,bp})$:
\begin{equation}\label{eqn-d-action-matrix-generator}
\dif\left(~
\begin{DGCpicture}[scale=0.8]
\DGCPLstrand[thk1](0,0)(1,1.15)[$^{ap}$]
\DGCPLstrand[thk1](2,0)(1,1.15)[$^{bp}$]
\DGCPLstrand[thk1](1,1.15)(1,1.85)
\DGCPLstrand[thk1](1,1.85)(0,3)[`$_{ap}$]
\DGCPLstrand[thk1](1,1.85)(2,3)[`$_{bp}$]
\end{DGCpicture}
~\right)=
0 .
\end{equation}
Furthermore, the $\dif$ action is extended to all elements in the diagrammatic basis by the Leibniz rule and the differential action on symmetric polynomials (see equation \eqref{eqn-dif-on-Schur}).

To proceed, we will consider the following $p$-complex which is defined to be
\begin{equation}
V_{a,b}:= \bigoplus_{\lambda\in P(bp,ap)}\Bbbk \pi_{\lambda},
\end{equation}
together with the differential action determined as in equation \eqref{eqn-dif-on-Schur}:
\[
\dif(\pi_{\lambda})=\sum_{\mu=\lambda+\square}\mathrm{C}(\square) \pi_{\mu}
\]
We will need to compute the slash cohomology of this $p$-complex, and to do so we introduce the following definition.

\begin{defn}\label{def-p-Lima-partition}
A partition in $P(bp,ap)$ is called \emph{$p$-Lima} if it is obtained from a partition $\nu$ in $P(b,a)$ by expanding each box of $\nu$ into a $p\times p$ cube.

For the ease of notation later,  the set of Lima partitions inside $P(bp,ap)$ will be denoted as
\begin{equation}\label{eqn-def-Lima}
LP(bp,ap)=\{\lambda \in P(bp,ap)|\lambda~\textrm{is $p$-Lima}\},
\end{equation}
so that $LP(bp,ap)$ is in one-one bijection with the set of partitions in $P(b,a)$. When $a\to \infty$, set $LP(bp):=LP(bp,\infty)$.
\end{defn} 

Below is an example of a $3$-Lima partition that is obtained from the partition $(2,1)$.
\[
\small\ytableaushort{\none}*{2,1}
\xrightarrow{\textrm{$3\times 3$ expand each box}}
\tiny\tiny\ytableaushort{\none}*{6,6,6,3,3,3}
\]
This notion was defined in \cite[Appendix 4.2]{EllisQ}, and will be used to compute the slash cohomology of $V_{a,b}$ in the next result.

\begin{lem}\label{lem-p-Lima}
The slash cohomology $\mH_/(V_{a,b})$ of the $p$-complex $V_{a,b}$ consists of a direct sum of one-dimensional $p$-complexes spanned by $p$-Lima partitions:
\[
\mH_/(V_{a,b})\cong \bigoplus_{\lambda\in LP(bp,ap)}\Bbbk\pi_\lambda.
\] In particular, $\mH_/(V_{a,b})$ has dimension ${a+b\choose a}$.
\end{lem}
\begin{proof}
By construction, $V_{a,b}$ embeds inside $\sym_{bp}$ as a $p$-subcomplex. We next define a projection map
\[
\sym_{bp}\lra V_{a,b},\quad
\pi_{\mu}\mapsto
\left\{
\begin{array}{ll}
\pi_{\mu} & \textrm{if}~\mu\in P(bp,ap),\\
0 & \textrm{otherwise}.
\end{array}
\right.
\]
If $\mu$ is a partition such that $\mu_1=ap$, and $\nu$ is obtained from $\mu$ by adding one box to the first row, then the content of the box added equals $ap=0\in \Bbbk$, so that $\pi_\nu$ does not figure in the differential of $\pi_\mu$. It then follows that the projection map commutes with the differentials. Thus $V_{a,b}$ is a $p$-complex direct summand of $\sym_{bp}$. 

By the differential action formula \eqref{eqn-dif-on-Schur}, it is clear that
\[
\{\pi_{\lambda}|\lambda \in LP(bp)\}
\]
is a family of $0$-cocycles for the $\dif$ action on $\sym_{bp}$. On the other hand, it has the same size as the monomial basis for the polynomial algebra $\Bbbk[e_{p}^p,\dots, e_{bp}^p]$, the latter being isomorphic to $\mH_/(\sym_{bp})$ by Lemma \ref{lemma-slash-homology-rank-one-mod}. An induction on the number of boxes via the Littlewood-Richardson rule shows that the monomial basis from $\Bbbk[e_{p}^p,\dots, e_{bp}^p]$ and the $p$-Lima Schur basis just differ by some null-homotopic terms. Therefore, we obtain
\[
\mH_/(\sym_{bp})\cong \bigoplus_{\lambda\in LP(bp)}\Bbbk \pi_{\lambda}.
\]
The Lemma then follows by truncating the partitions in $P(bp,\infty)$ onto the $p$-complex summand $P(bp,ap)$.
\end{proof}

\begin{rem}\label{rmk-categorifying-binomial-reduction} By Definition \ref{def-p-Lima-partition}, the number of $p$-Lima partitions $LP(bp,ap)$ is equal to the number of partitions inside $P(b,a)$. Therefore, computing the image of $V_{a,b}$ in the Grothendieck ring $K_0(\mc{D}^c(\Bbbk))\cong \mathbb{O}_p$ gives us
\[
q^{abp^2}{(a+b)p \brack ap}_{\mathbb{O}_p}=[V_{a,b}]_{\mathbb{O}_p}=
[\mH_/(V_{a,b})]_{\mathbb{O}_p}={a+b\choose a},
\]
which lies in the subring $\Z$ inside $\mathbb{O}_p$.
\end{rem}

\begin{prop}\label{prop-formality-end-algebra-rank-two}
The $p$-DG endomorphism algebra $\END_{\sym_{(a+b)p}}(\s_{ap,bp})$ is slash-zero formal, and it is quasi-isomorphic to a matrix algebra of size ${a+b \choose a}$ with coefficients in $\Bbbk[e_p^p,\dots, e_{(a+b)p}^p]$.
\end{prop}
\begin{proof}
Using Lemma \ref{lem-dif-basis-Grassmann-module}, we may rewrite the module $\s_{ap,bp}$ as
\[
\s_{ap,bp}\cong \sym_{(a+b)p}\otimes V_{a,b}\cong \sym_{(a+b)p}\otimes \mH_/(V_{a,b})\oplus \sym_{(a+b)p}\otimes P(V_{a,b}),
\]
where $P(V_{a,b})$ is a contractible $p$-complex and $\mH_/(V_{a,b})$ is a direct sum of trivial $p$-complexes by Lemma \ref{lem-p-Lima}.

Abbreviate the compact cofibrant $\sym_{(a+b)p}$-modules as 
\[
H:=\sym_{(a+b)p}\otimes \mH_/(V_{a,b}),\quad \quad P:=\sym_{(a+b)p}\otimes P(V_{a,b}).
\]
Then we may identify $\END_{\sym_{(a+b)p}}(\s_{ap,bp})$ as the block $p$-DG matrix algebra
\[
\END_{\sym_{(a+b)p}}(\s_{ap,bp})\cong 
\left(
\begin{matrix}
\HOM_{\sym_{(a+b)p}}(H,H) & \HOM_{\sym_{(a+b)p}}(H,P)\\
\HOM_{\sym_{(a+b)p}}(P,H) & \HOM_{\sym_{(a+b)p}}(P,P)
\end{matrix}
\right).
\]
It is clear that the terms $ \HOM_{\sym_{(a+b)p}}(H,P)$, $ \HOM_{\sym_{(a+b)p}}(P,H)$ and $ \HOM_{\sym_{(a+b)p}}(P,P)$ are acyclic since both $P$ and $H$ are cofibrant and $P$ is acyclic. Therefore, the natural inclusion
\[
\END_{\sym_{(a+b)p}}(H)\lra \END_{\sym_{(a+b)p}}(\s_{ap,bp})
\]
identifying the left hand term with the upper left block matrix is a quasi-isomorphism. Thus we are reduced to showing that $\END_{\sym_{(a+b)p}}(H)$ is slash-zero formal. But this is clear from Lemma \ref{lemma-slash-homology-rank-one-mod}: $\sym_{(a+b)p}$ is quasi-isomorphic to its slash cohomology ring, and thus
\[
\mH_/(\sym_{(a+b)p})\otimes \END_{\Bbbk}(\mH_/(V_{a,b}))\hookrightarrow \sym_{(a+b)p}\otimes \END_{\Bbbk}(\mH_/(V_{a,b})) \cong \END_{\sym_{(a+b)p}}(H) 
\]
is a chain of $p$-DG algebra quasi-isomorphisms.
\end{proof}

\begin{rem}
The cofibrance of $\s_{ap,bp}$ (and thus the cofibrance of its direct summands $H$ and $P$) over the $p$-DG algebra $\sym_{(a+b)p}$ has played a crucial role in the above proof. In general, the graded endormorphism space $\mathrm{END}_A(M)$ of a $p$-DG $A$-module $M$ will not control the endomorphism of $M$ in $\mc{D}(A,\dif)$ unless $M$ is cofibrant. For instance, consider the $p$-DG algebra $A=\Bbbk[x]$ with $\dif(x)=x^2$. The $p$-DG module $M=\Bbbk[x]\cdot v_0$ with
$
\dif(v_0):=xv_0
$
is clearly acyclic. It is not cofibrant, and it is easily verified that the induced differential on $\END_A(M)\cong \Bbbk[x]$, $f\mapsto f(v_0)$ is an isomorphism of $p$-DG algebras. However,
\[
\mH_{/0}^0(\END_A(M))\cong \Bbbk \neq \End_{\mc{D}(A,\dif)}(M)=0.
\]
\end{rem}

To give a diagrammatic description of the slash cohomology algebra elements, we just need to resort to the diagrammatic description \eqref{eqn-rank-one-mod-end-alg-diag} together with Lemma \ref{lem-p-Lima}. We have that the slash cohomology is spanned by the elements in the following set
\begin{equation}\label{eqn-rank-one-mod-end-alg-cohomology}
\left\{
(-1)^{|\hat{\mu}|}
\begin{DGCpicture}[scale=0.8]
\DGCPLstrand[thk1](0,0)(1,1.15)[$^{ap}$]
\DGCPLstrand[thk1](2,0)(1,1.15)[$^{bp}$]
\DGCPLstrand[thk1](1,1.15)(1,1.85)
\DGCPLstrand[thk1](1,1.85)(0,3)[`$_{ap}$]
\DGCPLstrand[thk1](1,1.85)(2,3)[`$_{bp}$]
\DGCcoupon(1.15,0.35)(1.85,0.85){$\scriptsize{\pi_{\hat{\mu}}}$}
\DGCcoupon(0.15,2.15)(0.85,2.65){$\scriptsize{\pi_\lambda}$}
\DGCcoupon(0.65,1.25)(1.35,1.75){$f$}
\end{DGCpicture}~\Bigg|
\lambda, \mu \in LP(ap,bp), f \in \Bbbk[e_p^p,\dots e_{(a+b)p}^p]
\right\}.
\end{equation}
This description, together with equation \eqref{eqn-composition-basis-elts}, also makes it clear that $\mH_/(\END_{\sym_{(a+b)p}}(\s_{ap,bp}))$ can be regraded as a subalgebra inside $\END_{\sym_{(a+b)p}}(\s_{ap,bp})$, although it does not contain the identity element.

When both $a=b=1$, the minimal nontrivial $p$-Lima partition is just $(p^p)$.
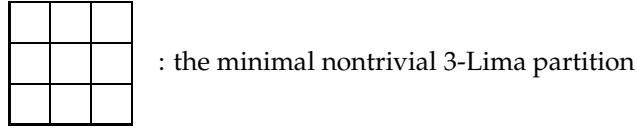
\[
\small\ytableaushort{\none}*{3,3,3}
\quad \textrm{: the minimal nontrivial 3-Lima partition}
\]
The associated Schur polynomial is just $e_p^p(x_1,\dots, x_p)$. The complement of $(p^p)$ is of course just the empty partition. Therefore, we have the following Corollary from Proposition \ref{prop-formality-end-algebra-rank-two}.

\begin{cor}The following relations hold on the slash cohomology of $\END_{\sym_{2p}}(\s_{p,p})$:
\begin{equation}\label{eqn-thick-nilHecke1}
\begin{DGCpicture}[scale=0.8]
\DGCPLstrand[thk1](0,0)(1,1.15)[$^{p}$]
\DGCPLstrand[thk1](2,0)(1,1.15)[$^{p}$]
\DGCPLstrand[thk1](1,1.15)(1,1.85)
\DGCPLstrand[thk1](1,1.85)(0,3)[`$_{p}$]
\DGCPLstrand[thk1](1,1.85)(2,3)[`$_{p}$]
%\DGCcoupon(1.15,0.3)(1.85,0.9){$\scriptsize{e_p^p}$}
\DGCcoupon(0.15,2.1)(0.85,2.7){$\scriptsize{e_p^p}$}
\end{DGCpicture}
-
\begin{DGCpicture}[scale=0.8]
\DGCPLstrand[thk1](0,0)(1,1.15)[$^{p}$]
\DGCPLstrand[thk1](2,0)(1,1.15)[$^{p}$]
\DGCPLstrand[thk1](1,1.15)(1,1.85)
\DGCPLstrand[thk1](1,1.85)(0,3)[`$_{p}$]
\DGCPLstrand[thk1](1,1.85)(2,3)[`$_{p}$]
\DGCcoupon(1.15,0.3)(1.85,0.9){$\scriptsize{e_p^p}$}
%\DGCcoupon(0.15,2.1)(0.85,2.7){$\scriptsize{e_p^p}$}
\end{DGCpicture}
=
\begin{DGCpicture}[scale=0.8]
\DGCPLstrand[thk1](0,0)(0,3)[$^{p}$`$_{p}$]
\DGCPLstrand[thk1](2,0)(2,3)[$^{p}$`$_{p}$]
\end{DGCpicture}
=
\begin{DGCpicture}[scale=0.8]
\DGCPLstrand[thk1](0,0)(1,1.15)[$^{p}$]
\DGCPLstrand[thk1](2,0)(1,1.15)[$^{p}$]
\DGCPLstrand[thk1](1,1.15)(1,1.85)
\DGCPLstrand[thk1](1,1.85)(0,3)[`$_{p}$]
\DGCPLstrand[thk1](1,1.85)(2,3)[`$_{p}$]
\DGCcoupon(0.15,0.3)(0.85,0.9){$\scriptsize{e_p^p}$}
%\DGCcoupon(0.15,2.1)(0.85,2.7){$\scriptsize{e_p^p}$}
\end{DGCpicture}
-
\begin{DGCpicture}[scale=0.8]
\DGCPLstrand[thk1](0,0)(1,1.15)[$^{p}$]
\DGCPLstrand[thk1](2,0)(1,1.15)[$^{p}$]
\DGCPLstrand[thk1](1,1.15)(1,1.85)
\DGCPLstrand[thk1](1,1.85)(0,3)[`$_{p}$]
\DGCPLstrand[thk1](1,1.85)(2,3)[`$_{p}$]
%\DGCcoupon(1.15,0.3)(1.85,0.9){$\scriptsize{e_p^p}$}
\DGCcoupon(1.15,2.1)(1.85,2.7){$\scriptsize{e_p^p}$}
\end{DGCpicture}
\ .
\end{equation}
\end{cor}
\begin{proof}
Only the second equality requires some comment since the definition of $\s_{c,d}$ was not symmetric in $c,d$ in the presence of the differential. But when both $c,d$ are divisible by $p$, the differential acts on the generator trivially (see equation \eqref{eqn-d-action-dual-mod-generator}), and thus interchanging $c$ and $d$ commutes with differentials. 
\end{proof}

Following \cite{KLMS}, we use a $p$-$p$ crossing diagram for simplification as follows.
\[
\begin{DGCpicture}[scale=0.75]
\DGCstrand[thk1](0,0)(1,1)[$^p$`$_p$]
\DGCstrand[thk1](1,0)(0,1)[$^p$`$_p$]
\end{DGCpicture}:=
\begin{DGCpicture}[scale=0.6]
\DGCPLstrand[thk1](0,0)(1,1.15)[$^{p}$]
\DGCPLstrand[thk1](2,0)(1,1.15)[$^{p}$]
\DGCPLstrand[thk1](1,1.15)(1,1.85)
\DGCPLstrand[thk1](1,1.85)(0,3)[`$_{p}$]
\DGCPLstrand[thk1](1,1.85)(2,3)[`$_{p}$]
\end{DGCpicture} \ .
\]
Now, notice that, for degree reasons, we have
\begin{equation}\label{eqn-thick-nilHecke2}
\begin{DGCpicture}[scale=0.75]
\DGCstrand[thk1](0,0)(1,1)(0,2)[$^p$`$_p$]
\DGCstrand[thk1](1,0)(0,1)(1,2)[$^p$`$_p$]
\end{DGCpicture}
=0
\end{equation}
in the algebra $\END_{\sym_{2p}}(\s_{p,p})$ even before taking slash cohomology. Furthermore, the Reidemeister-III type relation holds in the endomorphism algebra $\END_{\sym_{3p}}(\s_{p,p,p})$:
\begin{equation}\label{eqn-thick-nilHecke3}
\begin{DGCpicture}[scale=0.75]
\DGCstrand[thk1](0,0)(2,2)[$^p$`$_p$]
\DGCstrand[thk1](2,0)(0,2)[$^p$`$_p$]
\DGCstrand[thk1](1,0)(0,1)(1,2)[$^p$`$_p$]
\end{DGCpicture}
=
\begin{DGCpicture}[scale=0.75]
\DGCstrand[thk1](0,0)(2,2)[$^p$`$_p$]
\DGCstrand[thk1](2,0)(0,2)[$^p$`$_p$]
\DGCstrand[thk1](1,0)(2,1)(1,2)[$^p$`$_p$]
\end{DGCpicture}
\ ,
\end{equation}
which is a special case of \cite[Corollary 2.6]{KLMS}. Observe that the relations \eqref{eqn-thick-nilHecke1}--\eqref{eqn-thick-nilHecke3} are just the usual nilHecke algebra relations implemented on thickness-$p$ strands. Recall that, if $a\in \N$, the nilHecke algebra on $a$ letters is by definition $\mathrm{NH}_a\cong \END_{\sym_a}(\pol_a)$ \cite{KL1, Rou2}. It has a diagrammatic presentation given by local generators
\begin{equation}\label{eqn-nilHecke-generator}
\begin{DGCpicture}
\DGCstrand[thn2](0,0)(0,1.5)
\DGCdot{0.75}
\end{DGCpicture}
\quad \quad \quad \quad
\begin{DGCpicture}
\DGCstrand[thn2](0,0)(1.5,1.5)
\DGCstrand[thn2](1.5,0)(0,1.5)
\end{DGCpicture} \ ,
\end{equation}
which have degrees $2$ and $-2$ respectively, quotient the relations
\begin{equation}
\begin{DGCpicture}
\DGCstrand[thn2](0,0)(1,1)
\DGCstrand[thn2](1,0)(0,1)\DGCdot{0.75}
\end{DGCpicture}
-
\begin{DGCpicture}
\DGCstrand[thn2](0,0)(1,1)
\DGCstrand[thn2](1,0)(0,1)\DGCdot{0.25}
\end{DGCpicture}
=
\begin{DGCpicture}
\DGCstrand[thn2](0,0)(0,1)
\DGCstrand[thn2](1,0)(1,1)
\end{DGCpicture}
=
\begin{DGCpicture}
\DGCstrand[thn2](0,0)(1,1)\DGCdot{0.25}
\DGCstrand[thn2](1,0)(0,1)
\end{DGCpicture}
-
\begin{DGCpicture}
\DGCstrand[thn2](0,0)(1,1)\DGCdot{0.75}
\DGCstrand[thn2](1,0)(0,1)
\end{DGCpicture}
\ ,
\end{equation}
\begin{equation}
\begin{DGCpicture}[scale=0.75]
\DGCstrand[thn2](0,0)(1,1)(0,2)
\DGCstrand[thn2](1,0)(0,1)(1,2)
\end{DGCpicture}
=0,
\quad \quad \quad
\begin{DGCpicture}[scale=0.75]
\DGCstrand[thn2](0,0)(2,2)
\DGCstrand[thn2](2,0)(0,2)
\DGCstrand[thn2](1,0)(0,1)(1,2)
\end{DGCpicture}
=
\begin{DGCpicture}[scale=0.75]
\DGCstrand[thn2](0,0)(2,2)
\DGCstrand[thn2](2,0)(0,2)
\DGCstrand[thn2](1,0)(2,1)(1,2)
\end{DGCpicture}
\ .
\end{equation}

In order to make the next map homogeneous, we will adjust the above local generators \eqref{eqn-nilHecke-generator} of nilHecke algebras to be of degree $2p^2$ and $-2p^2$ respectively.

\begin{defn}\label{def-half-thicken}
For any $a\in \N$, the \emph{thickening map} 
\[
\Theta^+:(\mathrm{NH}_a,\dif_0\equiv 0)\lra (\END_{\sym_{ap}}(\s_{(p^a)}),\dif)
\]
is given locally on the diagrammatic generators of $\mathrm{NH}_a$ by
\[
\Theta^+\left(~
\begin{DGCpicture}
\DGCstrand[thn2](0,0)(0,1.5)
\DGCdot{0.75}
\end{DGCpicture}
~\right)
:=
\begin{DGCpicture}
\DGCstrand[thk1](0,0)(0,1.5)[$^p$`$_p$]
\DGCcoupon(-0.35,0.5)(0.35,1){$\scriptsize{e_p^p}$}
\end{DGCpicture} \ ,
\quad \quad \quad
\Theta^+\left(~
\begin{DGCpicture}
\DGCstrand[thn2](0,0)(1.5,1.5)
\DGCstrand[thn2](1.5,0)(0,1.5)
\end{DGCpicture}
~\right):=
\begin{DGCpicture}
\DGCstrand[thk1](0,0)(1.5,1.5)[$^p$`$_p$]
\DGCstrand[thk1](1.5,0)(0,1.5)[$^p$`$_p$]
\end{DGCpicture}
\ .
\]
\end{defn}

\begin{thm}
\label{thm-nil-Hecke-holds} The $p$-DG endomorphism algebra
$ (\END_{\sym_{ap}}(\s_{(p^a)}),\dif) $ is slash-zero formal. 
The thickening map $\Theta^+$ induces a quasi-isomorphism of $p$-DG algebras. 
\end{thm}
\begin{proof}
The nilHecke relations \eqref{eqn-thick-nilHecke1}--\eqref{eqn-thick-nilHecke3} have been established in the discussion before the Theorem, so this homomorphism is well-defined.
To prove that $\END_{\sym_{ap}}(\s_{(p^a)},\dif)$ is slash-zero formal, one uses a similar argument as in Proposition \ref{prop-formality-end-algebra-rank-two}. 

To do so, notice that the $p$-DG module
\[
\s_{(p^a)}= (\sym_p)^{\otimes a}\cdot v_0,\quad \quad \dif(v_0)=0
\]
is compact cofibrant over $\sym_{ap}$. An induction together with the $a=b=1$ case of Proposition \ref{prop-formality-end-algebra-rank-two} shows that
\[
\s_{(p^a)}\cong \sym_{ap}\otimes H_0 \oplus \sym_{ap}\otimes P_0
\]
where $H_0$ is a direct sum of trivial $p$-complexes that has total dimension $\mathrm{dim}(H_0)=a!$, and $P_0$ is an acyclic $p$-complex. It follows that
\[
\sym_{ap}\otimes \END_{\Bbbk}(H_0)\cong \END_{\sym_{ap}}(\sym_{ap}\otimes H_0)\hookrightarrow\END_{\sym_{ap}}(\s_{p^a})
\]
is a quasi-isomorphism of $p$-DG algebras. The slash-zero formality then follows from Lemma \ref{lemma-slash-homology-rank-one-mod}. To directly compare it with the nilHecke algebra on the slash cohomology level, we use that
\[
\iota:\sym_{ap}\lra (\sym_{p})^{\otimes a}
\]
identifies the slash cohomology of $\sym_{ap}$ with the $S_a$-invariant part of $(\mH_/(\sym_p))^{\otimes a}$ (Corollary \ref{cor-slash-cohomology-inclusion-2}). Since
$
(\mH_/(\sym_p))^{\otimes a}\cong \Bbbk[y_1,\dots, y_i]
$
with $y_i= e_p^p(x_{(i-1)p+1},\dots, x_{ip})$, $\mH_/(\sym_{ap})\cong \sym_a(y_1,\dots, y_a)$, and it follows that
\[
\mH_/(\END_{\s_{(p^a)}}) \cong \END_{\sym_a}(\pol_a)\cong \mathrm{NH}_a.
\]
This finishes the proof of the Theorem.
\end{proof}

\begin{rem}
 Theorem \ref{thm-nil-Hecke-holds} promotes Theorem \ref{thm-weak-categorification-half-sl2} into a strong categorification by considering the $p$-DG derived category
\[
\bigoplus_{a\in \N}\mc{D}(\END_{\sym_{ap}}(\s_{(a^p)}))
\]
with its enhanced $p$-DG endomorphism structure (Remark \ref{rmk-enhancement}). It is not hard to show that this category is Morita equivalent to $\mc{D}_{(p)}(\sym)$ (This will be shown in slight more generality in Lemma \ref{lem-reduction-2}.). Composing the Morita equivalence, the natural inclusion $\jmath$ of $\mc{D}_{(p)}(\sym)$ into $\mc{D}(\sym)$ together with $\Theta^+$ gives us a fully-faithful embedding of enhanced $p$-DG derived categories
\[
 \mc{D}(\mathrm{NH},\dif_0):= \bigoplus_{a\in \N}\mc{D}(\mathrm{NH}_a,\dif_0)
\lra \mc{D}(\sym).
\]
It then follows from Theorem \ref{thm-weak-categorification-half-sl2} and Theorem \ref{thm-nil-Hecke-holds} that the natural projection from $\mc{D}(\sym)$ onto $\mc{D}_{(p)}(\mathrm{Sym})\cong \mc{D}(\mathrm{NH},\dif_0)$ categorifies the quantum Frobenius map (Definition \ref{def-quantum-Frob-half}.)
\end{rem}

\begin{rem}
 We expect that a similar construction can be extended to the $\mathfrak{sl}_3$ case by combining Sto\v{s}i\'{c}'s $\mathfrak{sl}_3$-thick calculus (\cite{Stosicsl3} ) with the differential defined in \cite{KQ}, which will then categorify the quantum Frobenius map for the upper half of $\mathfrak{sl}_3$ at a prime root of unity. Likewise, a careful modification of the previous construction in the DG odd nilHecke algebra case \cite{EllisQ} will give rise to a characteristic-zero lifting of the quantum Frobenius map for $\mathfrak{sl}_2$ at a fourth root of unity.
\end{rem}

\begin{rem}\label{rem-weak-thicking-half}
As we have, the direct sum of the maps $\Theta^+$ will give us a \emph{weak $p$-DG monoidal functor}, which becomes a triangulated monoidal functor of enhanced derived categories. Here we use the notion ``weak" to indicate that the map is not a $p$-DG algebra homomorphism. For instance, the relations\eqref{eqn-thick-nilHecke1} of nilHecke algebras only hold in slash cohomology , but on the abelian level the equality only holds up to acyclic idempotents.  However, upon passing to slash-cohomology, the map becomes an isomorphism. Combining the fact that $\END_{\sym_{ap}}(\s_{(p^a)},\dif)$ is slash-zero formal (Theorem \ref{thm-nil-Hecke-holds}), the functors give a ``roof'' (see Remark \ref{rmk-roofs}) of $p$-DG monoidal categories, and hence induces (enhanced) derived equivalences.
\end{rem}

\section{Thick calculus and quantum Frobenius}\label{sec-full-sl2}
In this Section, we define a thickening 2-functor $\Theta$ which is a ``doubled" version of the map $\Theta^+$ from Definition \ref{def-half-thicken}. We will show that the restriction functor along $\Theta$ gives rise to a categorification of the quantum Frobenius map for an idempotented version of quantum $\mf{sl}_2$ at a prime root of unity.

\subsection{Quantum \texorpdfstring{$\mathfrak{sl}_2$}{sl(2)} at a root of unity}
We first recall the definition of some idempotented forms of the generic and root-of-unity quantum $\mf{sl}_2$ over the ring $\mathbb{O}_p$ following \cite[Chapter 36]{Lus4}. Another concise and lucid account can be found in \cite{McG}. Then we will introduce a modified version of Lusztig's quantum Frobenius map for $\mathfrak{sl}_2$, which is defined over the ring $\mathbb{O}_p$. 

\begin{defn}\label{def-big-sl2}
\begin{enumerate}
\item[(i)]The non-unital associative quantum algebra $\dot{U}_{\Q(v)}(\mathfrak{sl}_2)$ is a $\Q(v)$-algebra generated by a family of orthogonal idempotents $\{1_n|n \in \Z\}$, a raising operator $\theta$ and a lowering operator $\vartheta$, subject to the following conditions:
\begin{enumerate}
\item[(i.1)] $ 1_n \cdot 1_{m}=\delta_{n,m}1_{n}$ for any $n, m \in \Z$,
\item[(i.2)] $ \theta 1_{n} = 1_{n+2} \theta $, \quad $ \vartheta 1_{n} = 1_{n-2} \vartheta $,
\item[(i.3)] $ \theta \vartheta 1_{n}-\vartheta \theta1_{n} =[n]_v1_{n}$.
\end{enumerate}
We denote by $\dot{U}_{v}(\mathfrak{sl}_2)$ (or just $\dot{U}_{v}$ for short) the $\Z[v^{\pm}]$-integral subalgebra of $\dot{U}_{\Q(v)}(\mathfrak{sl}_2)$ generated by the \emph{divided power elements} 
\begin{equation}\label{eqn-div-element}
\theta^{(a)}1_n:= \frac{\theta^a 1_n}{[a]_v!}, \quad \quad \vartheta^{(a)}1_n:= \frac{\vartheta^a1_n}{[a]_v!},
\end{equation}
for every $n\in \Z$ and $a\in \N$. 
\item[(ii)] Let $\mathbb{O}_p=\Z[v^{\pm 1}]/(\Psi_p(v^2))$. The $\mathbb{O}_p$-integral form $\dot{U}_{\mathbb{O}_p}$ is obtained by base changing $\dot{U}_v$ along the canonical ring map $\Z[v^{\pm 1}]\lra \mathbb{O}_p$, $v\mapsto q$:
\[
\dot{U}_{\mathbb{O}_p}:=\dot{U}_v\otimes_{\Z[v^{\pm 1}]}\mathbb{O}_p.
\] 
The base changed divided power elements will be denoted by
\[
E^{(a)}1_n:={\theta}^{(a)}1_n\otimes 1,\quad \quad
F^{(a)}1_{n}:=\vartheta^{(a)}1_n\otimes 1.
\]
\item[(iii)] Let $\rho:\Z[v^{\pm 1}]\lra \mathbb{O}_p$ be the ring homomorphism that takes $v$ to $q^{p}$. The $\mathbb{O}_p$-integral idempotented algebra $\dot{U}_\rho$ is obtained from $\dot{U}_v$ along\footnote{As noted in Remark \ref{rem-wrong-grading-choice}, this is not quite the idempotented form for the universal enveloping algebra for $\mf{sl}_2$, but will need a further base change to the cyclotomic field $\mathcal{O}_p$.} $\rho$:
$$\dot{U}_\rho:=\dot{U}_v\otimes_{\Z[v^{\pm 1}],\rho}\mathbb{O}_p. $$
 The base changed divided power generators in this case will be written as
\[
\mathsf{E}^{(a)}1_n:=\theta^{(a)} 1_n\otimes_{\rho} 1, \quad \quad\mathsf{F}^{(a)}1_n:=\vartheta^{(a)} 1_n\otimes_{\rho}1.
\]
\item[(iv)]
The non-unital associative (small) quantum algebra $\dot{u}_{\mathbb{O}_p}:=\dot{u}_{\mathbb{O}_p}(\mathfrak{sl}_2)$ is the subalgebra of $\dot{U}_{\mathbb{O}_p}$ generated by $\{E1_n, F1_n|n\in \Z\}$. 
\end{enumerate}
\end{defn}

\begin{defn}\label{defn-canonical-basis} Lusztig's \emph{canonical basis} $\dot{\mathbb{B}}$ is an additive $\Z[v^{\pm 1}]$-basis for $\dot{U}_v$, which consists of
\begin{itemize}
\item $\theta^{(a)}\vartheta^{(b)}1_n$, where $a,b \in \N$, $n \in \Z$ with $n \leq b-a$;
\item $\vartheta^{(b)}\theta^{(a)}1_n$, where $a,b \in \N$, $n \in \Z$ with $n \geq b-a$,
\end{itemize}
together with the identifications $\theta^{(a)}\vartheta^{(b)}1_{b-a}=\vartheta^{(b)}\theta^{(a)}1_{b-a}$.
\end{defn}

Via base change, the canonical basis $\dot{\mathbb{B}}$ gives rise to bases for $\dot{U}_{\mathbb{O}_p}$ and $\dot{U}_\rho$, which we denote $\dot{\mathbb{B}}_{\mathbb{O}_p}$ and $\mathbb{\dot{B}}_\rho$ respectively.

\begin{defn}\label{def-quantum-Frob}
The \emph{quantum Frobenius map} for $\mathfrak{sl}_2$ at a $p$th root of unity is the idempotented algebra homomorphism defined on the generators by
\[
\mathrm{Fr}: \dot{U}_{\mathbb{O}_p} \lra \dot{U}_\rho,
\]
\[ 
E^{(a)}\1_{n}
\mapsto
\left\{
\begin{array}{ll}
\mathsf{E}^{(a/p)}\1_{n/p} & \mathrm{if}~p~|~a,n,\\
0 & \textrm{otherwise},
\end{array}
\right.
\quad \quad
F^{(a)}\1_{n}
\mapsto
\left\{
\begin{array}{ll}
\mathsf{F}^{(a/p)}\1_{n/p} & \mathrm{if}~p~|~a,n,\\
0 & \textrm{otherwise}.
\end{array}
\right.
\]
Here $\rho$ is the base change ring homomorphism \eqref{eqn-base-change-map}.
\end{defn} 

It is not hard to see that the kernel of $\mathrm{Fr}$ is generated by $\dot{u}_{\mathbb{O}_p}$ as an ideal inside $\dot{U}_{\mathbb{O}_p}$, and thus the quantum Frobenius map fits into the sequence of $\mathbb{O}_p$-algebras
\begin{equation}\label{eqn-Frobenius-sequence}
0 \lra \dot{u}_{\mathbb{O}_p}\lra \dot{U}_{\mathbb{O}_p} \stackrel{\mathrm{Fr}}{\lra}
\dot{U}_\rho \lra 0.
\end{equation}
The ideal generated by $\dot{u}_{\mathbb{O}_p}$ constitutes the kernel of $\mathrm{Fr}$.

\subsection{Thin and thick \texorpdfstring{$\mathfrak{sl}_2$}{sl(2)} calculus}

\paragraph{Thin calculus.} We now recall a version of the thin diagrammatic calculus of Lauda \cite{Lau1} that categorifies $\dot{U}_v$ at a generic value.

\begin{defn}\label{def-u-dot} The $2$-category $\mathcal{U}$ is an additive graded $\Bbbk$-linear category. It has one object for each $n \in \Z$, where $\Z$ is the weight lattice of $\mathfrak{sl}_2$. The $1$-morphisms are (direct sums of grading shifts of) composites of the generating $1$-morphisms $\1_{n+2} \mathcal{E} \1_n$ and $\1_n \mathcal{F} \1_{n+2}$, where $n$ ranges over $\Z$. Each $\1_{n+2} \mathcal{E} \1_n$ will be drawn the same, regardless of the object $n$. One may think that there is a single $1$-morphism $\mathcal{E}$ which increases the weight by $2$ (read from right to left), and a single $1$-morphism $\mc{F}$ that decreases the weight by $2$. We draw the $1$-morphisms as oriented strands.

\begin{align*}
\begin{tabular}{|c|c|c|}
	\hline
	$1$-\textrm{Morphism generator} &
	\begin{DGCpicture}
	\DGCstrand[thn2](0,0)(0,1)
	\DGCdot*>{0.5}
	\DGCcoupon*(0.1,0.25)(1,0.75){$^n$}
	\DGCcoupon*(-1,0.25)(-0.1,0.75){$^{n+2}$}
    \DGCcoupon*(-0.25,1)(0.25,1.15){}
    \DGCcoupon*(-0.25,-0.15)(0.25,0){}
	\end{DGCpicture}&
	\begin{DGCpicture}
	\DGCstrand[thn2](0,0)(0,1)
	\DGCdot*<{0.5}
	\DGCcoupon*(0.1,0.25)(1,0.75){$^{n+2}$}
	\DGCcoupon*(-1,0.25)(-0.1,0.75){$^n$}
    \DGCcoupon*(-0.25,1)(0.25,1.15){}
    \DGCcoupon*(-0.25,-0.15)(0.25,0){}
	\end{DGCpicture} \\ \hline
	\textrm{Name} & $\1_{n+2}\mathcal{E}\1_n$ & $\1_n\mathcal{F}\1_{n+2}$ \\
	\hline
\end{tabular}
\end{align*}

The $2$-morphisms are generated by the following pictures
with prescribed degrees\footnote{One should note that we have taken the liberty to expand the degrees of generating $2$-morphisms by $p^2$ systematically. This is to ensure that the functor $\Theta$ we will define is homogeneous.}.
	
\begin{align*}
\begin{tabular}{|c|c|c|c|c|}
  \hline
  \textrm{Generator} &
  \begin{DGCpicture}
  \DGCstrand[thn2](0,0)(0,1)
  \DGCdot*>{0.95}
  \DGCdot{0.45}
  \DGCcoupon*(0.1,0.25)(1,0.75){$^n$}
  \DGCcoupon*(-1,0.25)(-0.1,0.75){$^{n+2}$}
  \DGCcoupon*(-0.25,1)(0.25,1.15){}
  \DGCcoupon*(-0.25,-0.15)(0.25,0){}
  \end{DGCpicture}
  &
  \begin{DGCpicture}
  \DGCstrand[thn2](0,0)(1,1)
  \DGCdot*>{0.95}
  \DGCstrand[thn2](1,0)(0,1)
  \DGCdot*>{0.95}
  \DGCcoupon*(1.1,0.25)(2,0.75){$^n$}
  \DGCcoupon*(-1,0.25)(-0.1,0.75){$^{n+4}$}
  \DGCcoupon*(-0.25,1)(0.25,1.15){}
  \DGCcoupon*(-0.25,-0.15)(0.25,0){}
  \end{DGCpicture} 
  &
   \begin{DGCpicture}
  \DGCstrand[thn2]/d/(0,0)(1,0)
  \DGCdot*<{-0.25,2}
  \DGCcoupon*(1,-0.5)(1.5,0){$^n$}
  \DGCcoupon*(-0.25,0)(1.25,0.15){}
  \DGCcoupon*(-0.25,-0.65)(1.25,-0.5){}
  \end{DGCpicture}
  &
  \begin{DGCpicture}
  \DGCstrand[thn2](0,0)(1,0)/d/
  \DGCdot*<{0.25,1}
  \DGCcoupon*(1,0)(1.5,0.5){$^n$}
  \DGCcoupon*(-0.25,0.5)(1.25,0.65){}
  \DGCcoupon*(-0.25,-0.15)(1.25,0){}
  \end{DGCpicture}
   \\ \hline
  \textrm{Degree}  & $2p^2$   & $-2p^2$ & $(1-n)p^2$ & $(1+n)p^2$ \\
  \hline
\end{tabular}
\end{align*}

The following set of relations is imposed on the $2$-morphism generators.

\begin{enumerate}
\item[(i)]The \emph{nilHecke relations} are satisfied on upward pointing strands regardless of regions.
\begin{equation}
\begin{DGCpicture}
\DGCstrand[thn2](0,0)(1,1)
\DGCdot>{0.95}
\DGCstrand[thn2](1,0)(0,1)\DGCdot{0.65}
\DGCdot>{0.95}
\end{DGCpicture}
-
\begin{DGCpicture}
\DGCstrand[thn2](0,0)(1,1)
\DGCdot>{0.95}
\DGCstrand[thn2](1,0)(0,1)\DGCdot{0.25}
\DGCdot>{0.95}
\end{DGCpicture}
=
\begin{DGCpicture}
\DGCstrand[thn2](0,0)(0,1)
\DGCdot>{0.95}
\DGCstrand[thn2](1,0)(1,1)
\DGCdot>{0.95}
\end{DGCpicture}
=
\begin{DGCpicture}
\DGCstrand[thn2](0,0)(1,1)\DGCdot{0.25}
\DGCdot>{0.95}
\DGCstrand[thn2](1,0)(0,1)
\DGCdot>{0.95}
\end{DGCpicture}
-
\begin{DGCpicture}
\DGCstrand[thn2](0,0)(1,1)\DGCdot{0.65}\DGCdot>{0.95}
\DGCstrand[thn2](1,0)(0,1)\DGCdot>{0.95}
\end{DGCpicture}
\ ,
\end{equation}
\begin{equation}
\begin{DGCpicture}[scale=0.75]
\DGCstrand[thn2](0,0)(1,1)(0,2)
\DGCdot>{1.95}
\DGCstrand[thn2](1,0)(0,1)(1,2)
\DGCdot>{1.95}
\end{DGCpicture}
=0,
\quad \quad \quad
\begin{DGCpicture}[scale=0.75]
\DGCstrand[thn2](0,0)(2,2)\DGCdot>{1.95}
\DGCstrand[thn2](2,0)(0,2)\DGCdot>{1.95}
\DGCstrand[thn2](1,0)(0,1)(1,2)\DGCdot>{1.95}
\end{DGCpicture}
=
\begin{DGCpicture}[scale=0.75]
\DGCstrand[thn2](0,0)(2,2)\DGCdot>{1.95}
\DGCstrand[thn2](2,0)(0,2)\DGCdot>{1.95}
\DGCstrand[thn2](1,0)(2,1)(1,2)\DGCdot>{1.95}
\end{DGCpicture}
\ .
\end{equation}
\item[(ii)] The \emph{adjunction relations} hold, irrelevant of region labels:
\begin{align} \label{eqn-adjoint}
\begin{DGCpicture}[scale=0.75]
\DGCstrand[thn2](0,0)(0,1)(1,1)(2,1)(2,2)
\DGCdot*<{0.25,1}
\DGCdot*<{1.65,1}
\end{DGCpicture}
~=~
\begin{DGCpicture}[scale=0.75]
\DGCstrand[thn2](0,0)(0,2)
\DGCdot*<{0.35}
\end{DGCpicture}
\ ,
\qquad \qquad
\begin{DGCpicture}[scale=0.75]
\DGCstrand[thn2](2,0)(2,1)(1,1)(0,1)(0,2)
\DGCdot*>{0.35}
\DGCdot*>{1.75}
\end{DGCpicture}
~=~
\begin{DGCpicture}[scale=0.75]
\DGCstrand[thn2](0,0)(0,2)
\DGCdot*>{1.5}
\end{DGCpicture}
\ .
\end{align}
\item[(iii)] For each $n\in \Z$, let $\sigma_n$ be the $2$-morphism
\begin{equation}
\label{eqn-sideways-crossing}
\sigma_n=
\begin{DGCpicture}
\DGCstrand[thn2]/ur/(0,0)(1,1)/ur/
\DGCdot<{0.05}
\DGCstrand[thn2]/ul/(1,0)(0,1)/ul/
\DGCdot>{0.95}
\DGCcoupon*(1,0.25)(1.25,0.75){$n$}
\end{DGCpicture}
~:=~
\begin{DGCpicture}[scale=0.35]
\DGCstrand[thn2]/u/(2,2)(4,4)/u/
\DGCstrand[thn2]/u/(4,2)(2,4)/u/
\DGCstrand[thn2](2,4)(1,5)(0,4)/d/(0,0)/d/
\DGCdot>{4.05,1}
\DGCstrand[thn2]/d/(4,2)(5,1)(6,2)/u/(6,6)/u/
\DGCdot<{1.95,1}
\DGCstrand[thn2]/u/(2,0)(2,2)/u/
\DGCdot>{0.95}
\DGCstrand[thn2]/u/(4,4)(4,6)/u/
\DGCdot>{5.05}
\DGCcoupon*(7,0.5)(8,5.5){$n$}
\end{DGCpicture}
\ .
\end{equation}
Then the following morphisms should be invertible.
\begin{equation}
\label{eqn-invertible-element-1}
\begin{DGCpicture}
\DGCstrand[thn2]/ur/(0,0)(1,1)/ur/
\DGCdot<{0.05}
\DGCstrand[thn2]/ul/(1,0)(0,1)/ul/
\DGCdot>{0.95}
\DGCcoupon*(1,0.25)(1.25,0.75){$n$}
\end{DGCpicture}
+ \sum_{i=0}^{n-1}~
\begin{DGCpicture}[scale=0.6]
\DGCstrand[thn2]/d/(0,0)(1,-1)(2,0)/u/[${}$]
\DGCdot<{-0.5,1}
\DGCdot{-1}[dr]{$_i$}
\DGCcoupon*(2.25,-0.35)(2.75,-1){$n$}
\end{DGCpicture}
\quad 
(n\geq 0),
\quad 
\quad 
\quad
\begin{DGCpicture}
\DGCstrand[thn2]/ur/(0,0)(1,1)/ur/
\DGCdot<{0.05}
\DGCstrand[thn2]/ul/(1,0)(0,1)/ul/
\DGCdot>{0.95}
\DGCcoupon*(1,0.25)(1.25,0.75){$n$}
\end{DGCpicture}
+ \sum_{i=0}^{n-1}~
\begin{DGCpicture}[scale=0.6]
\DGCstrand[thn2]/u/(0,0)(1,1)(2,0)/d/[${}$]
\DGCdot<{0.5,1}
\DGCdot{1}[ul]{$_i$}
\DGCcoupon*(2.25,0.35)(2.75,1){$n$}
\end{DGCpicture}
\quad 
(n\leq 0).
\end{equation}
\end{enumerate}
\end{defn}

\begin{thm}[Khovanov-Lauda, Rouquier]\label{thm-catfn-sl2-generic}
The category $\mc{U}$ categorifies $U_v(\mathfrak{sl}_2)$ at a generic $v$:
\[
K_0(\mc{U}-\mathrm{mod})\cong U_v(\mathfrak{sl}_2).
\] 
Here $v$ is the decategorification of the grading shift functor on $\mc{U}-\mathrm{mod}$. The indecomposable left projective modules over $\mc{U}$ in the set
\[
\left\{
\mc{E}^{(a)}\mc{F}^{(b)}\1_n ,
\mc{F}^{(a)}\mc{E}^{(b)}\1_m|a,b\in \N,~n\leq b-a,~m \geq a-b
\right\}
\]
 categorify Lusztig's canonical basis elements $\dot{\mathbb{B}}$:
 \[
[\mc{E}^{(a)}\mc{F}^{(b)}\1_n]= \theta^{(a)}\vartheta^{(b)}\1_n ,\quad \quad
[\mc{F}^{(a)}\mc{E}^{(b)}\1_m]= \vartheta^{(a)}\theta^{(b)}\1_m, 
 \]
and the modules $\mc{E}^{(a)}\mc{F}^{(b)}\1_{b-a}$, $\mc{F}^{(b)}\mc{E}^{(a)}\1_{b-a}$ are isomorphic for any $a,b\in \N$.
\end{thm}

\begin{rem}[On the proofs of Theorem \ref{thm-catfn-sl2-generic}]
The original proof of the result using a few more generators and relations is due to Khovanov-Lauda \cite{Lau1,KL3}. The relations presented here are defined by Rouquier \cite{Rou2}. The work of Cautis and Lauda \cite{CauLau} shows that the two definitions are closely related. More recently, it is shown by Brundan in the amazing work \cite{Brundan2KM} that the two presentations of Khovanov-Lauda and Rouquier are equivalent, significantly reducing the amount of relations to check in many situations.

The second part of the Theorem, regarding lifting canonical basis elements to indecomposable $\mc{U}$-modules, can be found in \cite[Section 5]{KLMS}.
\end{rem}

\begin{cor}\label{cor-catfn-classical-enveloping-algebra}
Equip $\mc{U}$ with the zero $p$-differential $\dif_0\equiv 0$. Then the compact derived category of $p$-DG modules over $\mc{U}$
categorifies $\dot{U}_\rho$:
\[
K_0(\mc{D}^c(\mc{U},\dif_0))\cong \dot{U}_\rho.
\]
\end{cor}
\begin{proof} By our choice of grading in Definition \ref{def-u-dot}, the $2$-morphism generators all have degrees expanded by $p^2$. On the level of Grothendieck group $K_0(\mc{D}^c(\mc{U}),\dif_0)$, this has the effect of specializing all the structural constants involving $v$ in Theorem \ref{thm-catfn-sl2-generic} at $q^{p^2}=q^p\in \mathbb{O}_p$. The isomorphism follows then, for instance, by using Corollary \ref{cor-formal-pDGA-K-group}.
\end{proof}

\paragraph{Thick calculus under differential.} Let us briefly recall a $p$-DG  $2$-category structure on $\mc{\dot{U}}$ defined in \cite{EQ2}. The category $\mc{\dot{U}}$ has been introduced in \cite{KLMS} to describe diagrammatically the Karoubi envelope of $\mc{U}$. As a result, $\mc{\dot{U}}$ is Morita equivalent to $\mc{U}$ and thus categorifies $\dot{U}_v$ as in Theorem \ref{thm-catfn-sl2-generic}. In the presence of a nontrivial differential, the derived category of $(\dot{\mc{U}},\dif)$ has been shown \cite[Theorem 6.2]{EQ2} to categorify the $\mathbb{O}_p$-integral $\dot{U}_{\mathbb{O}_p}$ with divided powers.

\begin{defn}\label{def-u-dot-big}
The $p$-DG $2$-category  $(\dot{\mc{U}},\dif)$ has the underlying $2$-category equivalent to the Karoubi envelope of $\mc{U}$ (Definition \ref{def-u-dot}). It has objects labeled by $n\in \Z$. The $1$-morphisms are monoidally generated by $\1_{n+2a}\mc{E}^{(a)}\1_n$, $\1_{n-2a}\mc{F}^{(a)}\1_n$ for all $a\in \N$ and $n\in \Z$. They are diagrammatically depicted by oriented thick strands of thickness $a$:
\[
\begin{DGCpicture}
\DGCstrand[thk1](0,0)(0,1)[$^a$`]
\DGCdot>{0.5}
\DGCcoupon*(0.1,0.4)(0.8,0.7){$_{n}$}
\DGCcoupon*(-0.2,0.4)(-1.2,0.7){$_{n+2a}$}
\end{DGCpicture}
\quad \quad \quad
\begin{DGCpicture}
\DGCstrand[thk1](0,0)(0,1)[$^a$`]
\DGCdot<{0.5}
\DGCcoupon*(0.1,0.4)(0.8,0.7){$_{n}$}
\DGCcoupon*(-0.2,0.4)(-1.2,0.7){$_{n-2a}$}
\end{DGCpicture}
\ .
\]
The $2$-morphisms are generated by thick diagrams below subject to certain relations that are similar to thin calculus ones. The reader is referred to \cite{KLMS} for the precise relations, but a $p$-DG $2$-category structure on $\dot{\mc{U}}$ is defined by declaring the following differential action on the $2$-morphism generators.

On upward pointing \emph{splitter} and \emph{merger} diagrams, the differential acts by  
\begin{equation}\label{eqn-dif-thick-upwards-splitter}
\dif \left(~
\begin{DGCpicture}[scale=0.75]
\DGCstrand[thk1](0,0)(0,1)[$^{a+b}$]
\DGCstrand[thk1]/ul/(0,1)(-1,2)/ul/[`$_a$]
\DGCdot>{1.95}
\DGCstrand[thk1]/ur/(0,1)(1,2)/ur/[`$_b$]
\DGCdot>{1.95}
\end{DGCpicture}~
\right)
=-b
\begin{DGCpicture}[scale=0.75]
\DGCstrand[thk1](0,0)(0,1)[$^{a+b}$]
\DGCstrand[thk1]/ul/(0,1)(-1,2)/ul/[`$_a$]
\DGCdot>{1.95}
\DGCstrand[thk1]/ur/(0,1)(1,2)/ur/[`$_b$]
\DGCdot>{1.95}
\DGCcoupon(-0.8,1.25)(-0.2,1.65){$\scriptstyle{e_1}$}
\end{DGCpicture}
\ ,
\quad \quad \quad
\dif \left(~
\begin{DGCpicture}[scale=0.75]
\DGCPLstrand[thk1](0,0)(0,-1)[$_{a+b}$]
\DGCdot<{0}
\DGCPLstrand[thk1](0,-1)(-1,-2)[`$^a$]
\DGCPLstrand[thk1](0,-1)(1,-2)[`$^b$]
\end{DGCpicture}~
\right)
=-a
\begin{DGCpicture}[scale=0.75]
\DGCPLstrand[thk1](0,0)(0,-1)[$_{a+b}$]
\DGCdot<{0}
\DGCPLstrand[thk1](0,-1)(-1,-2)[`$^a$]
\DGCPLstrand[thk1](0,-1)(1,-2)[`$^b$]
\DGCcoupon(0.2,-1.3)(0.8,-1.7){$\scriptstyle{e_1}$}
\end{DGCpicture}
\ .
\end{equation}
On oriented \emph{biadjunction maps}, the differential is given by
\begin{equation}\label{eqn-dif-thick-cup-cap-1}
\dif\left(~
\begin{DGCpicture}[scale =0.5]
\DGCstrand[thk1]/u/(0,0)(0.75,1)(1.5,0)/d/
\DGCdot<{0.5,1}[r]{$\scriptstyle{a}$}
\DGCcoupon*(1.5,0.8)(1.7,1){$n$}
\end{DGCpicture}
~\right)
=
(n+a)
\begin{DGCpicture}[scale =0.5]
\DGCstrand[thk1]/u/(0,0)(0.75,1)(1.5,0)/d/
\DGCdot<{0.5,1}[r]{$\scriptstyle{a}$}
\DGCcoupon*(1.5,0.8)(1.7,1){$n$}
\DGCcoupon(1.2,0.2)(1.7,0.5){$\scriptstyle{e_1}$}
\end{DGCpicture}
\ ,
\quad \quad
\dif\left(~
\begin{DGCpicture}[scale =0.5]
\DGCstrand[thk1]/d/(0,0)(0.75,-1)(1.5,0)/u/
\DGCdot<{-0.5,1}[r]{$\scriptstyle{a}$}
\DGCcoupon*(1.5,-0.8)(1.7,-1){$n$}
\end{DGCpicture}
~\right)
=(a-n)
\begin{DGCpicture}[scale =0.5]
\DGCstrand[thk1]/d/(0,0)(0.75,-1)(1.5,0)/u/
\DGCdot<{-0.5,1}[r]{$\scriptstyle{a}$}
\DGCcoupon*(1.5,-0.8)(1.7,-1){$n$}
\DGCcoupon(1.2,-0.2)(1.7,-0.5){$\scriptstyle{e_1}$}
\end{DGCpicture}
\ ,
\end{equation}

\begin{equation}\label{eqn-dif-thick-cup-cap-2}
\dif\left(~
\begin{DGCpicture}[scale =0.5]
\DGCstrand[thk1]/u/(0,0)(0.75,1)(1.5,0)/d/
\DGCdot>{0.5,1}[r]{$\scriptstyle{a}$}
\DGCcoupon*(1.5,0.8)(1.7,1){$n$}
\end{DGCpicture}
~\right)
=
a~~~
\begin{DGCpicture}[scale =0.5]
\DGCstrand[thk1]/u/(0,0)(0.75,1)(1.5,0)/d/
\DGCdot>{0.5,1}[r]{$\scriptstyle{a}$}
\DGCcoupon*(1.5,0.8)(1.7,1){$n$}
\DGCcoupon(1.2,0.2)(1.7,0.5){$\scriptstyle{e_1}$}
\end{DGCpicture}
-a~~~
\begin{DGCpicture}[scale =0.5]
\DGCstrand[thk1]/u/(0,0)(0.75,1)(1.5,0)/d/
\DGCdot>{0.5,1}[r]{$\scriptstyle{a}$}
\DGCcoupon*(1.5,0.8)(1.7,1){$n$}
\DGCcoupon(1.2,0.2)(1.7,0.5){$\scriptstyle{e_1}$}
\DGCbubble[thn1](2.25,0.65){0.3}
\DGCdot<{0.65,L}
\DGCcoupon*(2,0.5)(2.5,0.8){$\scriptsize{1}$}
\end{DGCpicture}
\ ,
\end{equation}

\begin{equation}\label{eqn-dif-thick-cup-cap-3}
\dif\left(~
\begin{DGCpicture}[scale =0.5]
\DGCstrand[thk1]/d/(0,0)(0.75,-1)(1.5,0)/u/
\DGCdot>{-0.5,1}[r]{$\scriptstyle{a}$}
\DGCcoupon*(1.5,-0.8)(1.7,-1){$n$}
\end{DGCpicture}
~\right)
=
a~~~
\begin{DGCpicture}[scale =0.5]
\DGCstrand[thk1]/d/(0,0)(0.75,-1)(1.5,0)/u/
\DGCdot>{-0.5,1}[r]{$\scriptstyle{a}$}
\DGCcoupon*(1.5,-0.8)(1.7,-1){$n$}
\DGCcoupon(1.2,-0.2)(1.7,-0.5){$\scriptstyle{e_1}$}
\end{DGCpicture}
+a~~~
\begin{DGCpicture}[scale =0.5]
\DGCstrand[thk1]/d/(0,0)(0.75,-1)(1.5,0)/u/
\DGCdot>{-0.5,1}[r]{$\scriptstyle{a}$}
\DGCcoupon*(1.5,-0.8)(1.7,-1){$n$}
\DGCcoupon(1.2,-0.2)(1.7,-0.5){$\scriptstyle{e_1}$}
\DGCbubble[thn1](2.25,-0.65){0.3}
\DGCdot<{-0.65,L}
\DGCcoupon*(2,-0.5)(2.5,-0.8){$\scriptsize{1}$}
\end{DGCpicture}
\ ,
\end{equation}
Here the clockwise thickness-$1$ ``bubble''
\begin{equation}\label{equ-degree-2-bubble}
\begin{DGCpicture}
\DGCbubble[thn1](2.25,0.65){0.5}
\DGCdot<{0.65,L}
\DGCcoupon*(2,0.5)(2.5,0.8){$\scriptsize{1}$}
\end{DGCpicture}
:=
\begin{DGCpicture}
\DGCbubble[thn1](0,0){0.5}
\DGCdot<{0,L}
\DGCdot{-0.25,R}[r]{$_n$}
\DGCdot*.{0.25,R}[r]{$n$}
\end{DGCpicture}
\end{equation}
composed of a cup, $n$ dots and a cap, is an element lying inside $\END_{\mc{\dot{U}}}(\1_n)\cong \sym$. The differential action on the $\END_{\mc{\dot{U}}}(\1_n)$ is then determined, as before, according to equation \eqref{eqn-dif-on-Schur}.
\end{defn}

 For any $a,b\in \N$, $i\in \{0,\dots, \mathrm{min}(a,b)\}$, $n\in \Z$ and $\alpha\in P(i,n-i)$, let
\begin{equation}\label{eqn-element-phi-i}
\phi_{\alpha}^{i}
:=
\begin{DGCpicture}[scale=0.3]
\DGCstrand[thk1]/d/(0,8)(0,6)(2,4)(4,6)(4,8)/u/[$_a$`$_b$]
\DGCdot.{5.5,1}[r]{$_i$}
\DGCdot<{7.5,1}
\DGCdot<{7.3,2}
\DGCstrand[thk1]/d/(0,6)(-2,2)(2,0)(4,-2)/d/[`$^{a-i}$]
\DGCdot<{-1.1}
\DGCstrand[thk1]/d/(4,6)(6,2)(2,0)(0,-2)/d/[`$^{b-i}$]
\DGCdot>{-1.4}
\DGCcoupon(1,3)(3,5){$\pi_\alpha$}
\DGCcoupon*(5,6)(7,8){$n$}
\end{DGCpicture}
\in 
\HOM_{\dot{\mc{U}}}(\mc{E}^{(a)}\mc{F}^{(b)}\1_n, \mc{E}^{(a-i)}\mc{F}^{(b-i)}\1_n).
\end{equation}
In particular, when $i=0$, then $\alpha=\emptyset$, and
\begin{equation}\label{eqn-element-phi-0}
\phi^0:=
\begin{DGCpicture}
\DGCstrand[thk1]/ur/(0,0)(1,1)/ur/[$^b$`$_b$]
\DGCdot<{0.05}
\DGCstrand[thk1]/ul/(1,0)(0,1)/ul/[$^a$`$_a$]
\DGCdot>{0.95}
\DGCcoupon*(1,0.25)(1.25,0.75){$n$}
\end{DGCpicture}
~:=~
\begin{DGCpicture}[scale=0.5]
\DGCstrand[thk1]/u/(2,2)(4,4)/u/
\DGCstrand[thk1]/u/(4,2)(2,4)/u/
\DGCstrand[thk1](2,4)(1,5)(0,4)/d/(0,0)/d/[`$^b$]
\DGCdot>{1,1}
\DGCstrand[thk1]/d/(4,2)(5,1)(6,2)/u/(6,6)/u/[`$_b$]
\DGCdot<{5,1}
\DGCstrand[thk1]/u/(2,0)(2,2)/u/[$^a$]
\DGCdot>{0.95}
\DGCstrand[thk1]/u/(4,4)(4,6)/u/[`$_a$]
\DGCdot>{5.05}
\DGCcoupon*(7,3.5)(7.5,4.5){$n$}
\end{DGCpicture}
\in \HOM_{\dot{\mc{U}}}(\mc{E}^{(a)}\mc{F}^{(b)}\1_n, \mc{F}^{(b)}\mc{E}^{(a)}\1_n),
\end{equation}

The reason that we are interested in these special elements in $\dot{\mc{U}}$ is that they consititute ``half'' of the maps in the identity decomposition formula for $\END_{\mc{\dot{U}}}(\mc{E}^{(a)}\mc{F}^{(b)}\1_n)$ (Sto\v{s}i\'{c}'s formula, \cite[Theorem 5.6]{KLMS}). More precisely, there exist
\[
\psi^0 \in \HOM_{\dot{\mc{U}}}(\mc{F}^{(b)}\mc{E}^{(a)}\1_n,\mc{E}^{(a)}\mc{F}^{(b)}\1_n), \quad \quad
\psi_\alpha^i\in \HOM_{\dot{\mc{U}}}(\mc{E}^{(a-i)}\mc{F}^{(b-i)}\1_n,\mc{E}^{(a)}\mc{F}^{(b)}\1_n),
\]
such that $\{\psi_\alpha^i\psi_\alpha^i|0\leq i \leq \mathrm{min}(a,b), \alpha\in P(i,n+a-b-i)\}$ forms a complete family of pairwise orthogonal idempotents in $\END_{\dot{\mc{U}}}(\mc{E}^{(a)}\mc{F}^{(b)}\1_n)$:
\begin{equation}\label{eqn-Stosic-formula-algebraic}
\Id_{\mc{E}^{(a)}\mc{F}^{(b)}\1_n}=\psi^0\phi^0+\sum_{i=1}^{\mathrm{min}(a,b)}\sum_{\alpha\in P(i,n+a-b-i)}\psi_\alpha^i\phi_\alpha^i.
\end{equation}

The next Theorem is proven in \cite[Theorem 6.2]{EQ2}.

\begin{thm}\label{thm-big-sl2} The (enhanced) $p$-DG derived category $\mc{D}(\dot{\mc{U}},\dif)$ categorifies the idempotented $\mathbb{O}_p$-integral quantum group $\dot{U}_{\mathbb{O}_p}$:
\[
K_0(\mc{D}^c(\dot{\mc{U}},\dif))\cong \dot{U}_{\mathbb{O}_p}.
\]
Furthermore, the symbols of the compact cofibrant $p$-DG modules in
\[
\left\{
\mc{E}^{(a)}\mc{F}^{(b)}\1_n,\mc{F}^{(a)}\mc{E}^{(b)}\1_m|n\leq b-a, m\geq a-b
\right\}
\]
are identified with the specialization of Lusztig's canonical basis $\dot{\mathbb{B}}_{\mathbb{O}_p}$. \hfill $\square$
\end{thm}

We deduce some consequences of Theorem \ref{thm-big-sl2}, which should have been included in \cite{EQ2}.

First off, observe that, without the differential $\dif$, the $2$-category $\mc{\dot{U}}$ is Morita equivalent to $\mc{U}$, so that it also categorifies the generic quantum group $\dot{U}_v$ as in Theorem \ref{thm-catfn-sl2-generic}. Consequently, in the usual derived category $\mc{D}(\dot{\mc{U}}\!-\!\mathrm{mod})$ of modules over $\mc{\dot{U}}$, the objects $\mc{E}\1_n$ $\mc{F}\1_n$, when $n$ ranges over $\Z$, monoidally generate the entire derived category (under $1$-morphism compositions). However, in the presence of $\dif$, we need an extra family of $p$-DG modules, besides $\mc{E}\1_n$ $\mc{F}\1_n$, to generate $\mc{D}(\mc{\dot{U}},\dif)$, namely those in
\begin{equation}
\left\{\mc{E}^{(p)}\1_n, \mc{F}^{(p)}\1_n|n \in \Z\right\}.
\end{equation} 

To show this we do a few reductions.

\begin{lem} \label{lem-reduction-1}
Let $m$ be a natural number, and fix a decomposition $m=ap+r$ such that $a\in \N$ and $r\in \{0,\dots, p-1\}$. The objects $\mc{E}^{(m)}\1_n$ (resp.~$\mc{F}^{(m)}\1_n$) with $n\in \Z$ in the enhanced derived category $\mc{D}(\dot{\mc{U}},\dif)$ are contained in the triangulated subcategory monoidally generated by $\mc{E}^{(ap)}\1_{n+2r}$ and $\mc{E}^{(r)}\1_{n}$ (resp. $\mc{F}^{(ap)}\1_{n-2r}$ and $\mc{F}^{(r)}\1_{n}$). 
\end{lem}
\begin{proof} We will only prove the result for $\mc{E}$'s, the corresponding proof for $\mc{F}$'s is similar.

To do this, it amounts to show that $\mc{E}^{(m)}\1_n$ is contained in the smallest triangulated subcategory in $\mc{D}(\dot{\mc{U}},\dif)$ that contains $\mc{E}^{(ap)}\mc{E}^{(r)}\1_n$.

Note that, in $(\mc{\dot{U}},\dif)$, the endomorphism algebra of any $1$-morphism $\mc{E}^{(m)}\1_n$ is controlled by $\sym_m$ with the natural differential \eqref{eqn-dif-elementary-function} that is independent of the weight $n$. Likewise, the endomorphism space of $\mc{E}^{(ap)}\mc{E}^{(r)}\1_n$ is controlled, irrelevant of $n$, by the $p$-DG endomorphism algebra $\END_{\sym_{m}}(\s_{ap,r}^\vee)$ \footnote{See Definition \ref{def-Grassmannian-module} for the left $p$-DG $\sym_m$-module $\s_{ap,r}$. Here $\s_{ap,r}^\vee=\HOM_{\sym_m}(\s_{ap,r},\sym_m)$ is the dual module with $\sym_m$ acting on the right. A diagrammatic $p$-DG generator for $\s_{ap,r}^\vee$ is depicted as the splitter in \eqref{eqn-dif-thick-upwards-splitter}.}.
By Lemma \ref{lem-dif-basis-Grassmann-module}, $\s_{ap,r}$ is compact cofibrant over $\sym_m$ and so is its dual, so that 
$$\END_{\sym_{m}}(\s_{ap,r}^\vee)\cong \END_{\sym_{m}}(\s_{ap,r})\cong \s_{ap,r}^\vee\otimes_{\sym_m}\s_{ap,r}$$ 
is isomorphic to a matrix $p$-DG algebra with coefficients in $\sym_m$. Therefore, we have a derived tensor functor
\[
\s_{ap,r}^\vee\otimes^{\mathbf{L}}_{\sym_{m}}(-): \mc{D}(\sym_{m},\dif)\lra \mc{D}(\END_{\sym_{m}}(\s_{ap,r}^\vee),\dif)
\]
One may regard $\s_{ap,r}^\vee$ as a column (left) module over a matrix presentation of $\END_{\sym_{m}}(\s_{ap,r}^\vee)$, while $\s_{ap,r}$ as a row (right) module. We now resort to a $p$-DG Morita criterion established in \cite[Proposition 8.8]{QYHopf}, which tells us that the functor induces an equivalence of $p$-DG derived categories if $\s_{ap,r}^\vee$ is cofibrant as a $p$-DG algebra $\END_{\sym_{m}}(\s_{ap,r}^\vee)$. Once this is shown, the lemma follows.

Now we establish the cofibrance condition. By Definition \ref{def-Grassmannian-module}, $\dif$ acts on the module generator $v_{ap,r}$ of $\s_{ap,r}$ trivially, and thus
$$\sym_{m}\cdot v_{ap,r}\subset \s_{ap,r}$$ 
is a $p$-DG direct summand. Tensoring this split inclusion with $\s_{ap,r}^\vee$ over $\sym_m$ gives us a split inclusion
 \[
\s_{ap,r}^\vee\cong \s_{ap,r}^\vee \otimes_{\sym_{m}} \sym_m\cdot v_{ap,r}\subset \s_{ap,r}^\vee\otimes_{\sym_m}\s_{ap,r}\cong \END_{\sym_m}(\s_{ap,r}^\vee).
 \]
 Therefore, the module $p$-DG module $\s_{ap,r}^\vee$ is a $p$-DG direct summand of the natural rank-one free $p$-DG module over $\END_{\sym_m}(\s_{ap,r}^\vee)$. The cofibrance of $\s_{ap,r}^\vee$ now follows.
\end{proof}

\begin{lem}\label{lem-reduction-2}
Let $a$ be any positive natural number and $r$ be in the set $\{0,1,\dots, p-1\}$.
\begin{enumerate}
\item[(i)] For each $n\in \Z$, the module $\mc{E}^{(ap)}\1_n$ in $\mc{D}(\dot{\mc{U}},\dif)$ is contained in the triangulated subcategory monoidally generated by $\mc{E}^{(p)}\1_k$ with $k\in \{n, n+2p, \dots, n+ap\}$.
\item[(ii)] For each $n\in \Z$, the module $\mc{E}^{(r)}\1_n$ in $\mc{D}(\dot{\mc{U}},\dif)$ is contained in the triangulated subcategory monoidally generated by $\mc{E}\1_k$ with $k\in \{n, n+2,\dots, n+2r\}$.
\end{enumerate}
\end{lem}
\begin{proof}
The proof is similar to the previous Lemma. For instance, to show (i), one considers the $p$-DG algebras $\sym_{ap}$ and $\END_{\sym_{ap}}(\s_{(p^a)}^\vee)$ (Definition \ref{def-Grassmannian-module}). Then one is further reduced to showing that $\s_{(p^a)}^\vee$ is cofibrant over the endomorphism algebra $\END_{\sym_{ap}}(\s_{(p^a)}^\vee)$. We leave the details as exercises to the reader.
\end{proof}

Now combining these two Lemmas, we see that any $p$-DG module of the form $\mc{E}^{(m)}\1_n$ or $\mc{F}^{(m)}\1_n$, where $m\in \N$ and $n\in \Z$, is contained in the subcategory generated by the modules
$\mc{E}\1_k$, $\mc{E}^{(p)}\1_k$,  $\mc{F}\1_k$, $\mc{F}^{(p)}\1_k$, where $k$ ranges over the weight lattice $\Z$. In particular, applying this observation to the ``canonical basis modules'' of Theorem \ref{thm-big-sl2}, we have proved the following. 

\begin{thm}\label{thm-4-generators}
The (enhanced) $p$-DG derived category $\mc{D}(\dot{\mc{U}},\dif)$ is monoidally generated by the compact cofibrant modules in the collection
\[
\left\{
\mc{E}\1_k,\mc{E}^{(p)}\1_k,\mc{F}\1_k,\mc{F}^{(p)}\1_k|k\in \Z
\right\}.
\]
On the Grothendieck group level, $\dot{U}_{\mathbb{O}_p}$ is generated by elements of the form $E1_k$, $E^{(p)}1_k$, $F1_k$, $F^{(p)}1_k$, where $k$ ranges over the weight lattice $\Z$. \hfill $\square$
\end{thm}

\begin{rem}
Of course the Theorem is false if we omit the $p$th divided power generators. This is because $\mc{E}^p\1_k$ is controlled by $(\mathrm{NH}_p,\dif)$, which is an acyclic $p$-DG algebra (see \cite[Section 3.3]{KQ}). It follows that $\mc{E}^{p}\1_k\cong 0$ in $\mc{D}(\mc{\dot{U}},\dif)$, and $\mc{E}^{(p)}\1_k$ could not be a summand of $\mc{E}^p\1_k$. In fact, a closer check of the proof of Lemmas \ref{lem-reduction-1} and \ref{lem-reduction-2} shows that any column module, which we denote as $\mc{P}$, over $(\mathrm{NH}_p,\dif)$ is not cofibrant, so that we do not obtain a derived equivalence
\[
\mc{P}\otimes_{\sym_p}^{\mathbf{L}}(-):\mc{D}(\sym_p)\lra \mc{D}(\mathrm{NH}_p).
\]
\end{rem}

\subsection{Categorification of quantum Frobenius for \texorpdfstring{$\mathfrak{sl_2}$}{sl(2)}}
Let us first define a doubled version $\Theta$ of the thickening map $\Theta^+$ in Theorem \ref{thm-nil-Hecke-holds}. 

\begin{defn}\label{def-thicken-functor-full}
Let $\mc{U}$ be equipped with the zero differential $\dif_0\equiv 0$, while the $p$-DG structure on $\mc{\dot{U}}$ is prescribed by Definition \ref{def-u-dot-big}. The (weak) \emph{thickening $2$-functor} $\Theta$ is the morphism of $p$-DG $2$-categories 
\[
\Theta:(\mc{U},\dif_0) \lra (\mc{\dot{U}},\dif)
\]
 defined as follows.

\begin{itemize}
\item[(i)] On the collection of $1$-morphism generators, let
\[
\Theta
\left(~
\begin{DGCpicture}
	\DGCstrand[thn2](0,0)(0,1)
	\DGCdot*>{0.95}
	\DGCcoupon*(0.1,0.25)(1,0.75){$^n$}
	\DGCcoupon*(-1,0.25)(-0.1,0.75){$^{n+2}$}
    \DGCcoupon*(-0.25,1)(0.25,1.15){}
    \DGCcoupon*(-0.25,-0.15)(0.25,0){}
	\end{DGCpicture}
~\right)
:=
\begin{DGCpicture}
	\DGCstrand[thk1](0,0)(0,1)[$^p$`$_p$]
	\DGCdot*>{0.95}
	\DGCcoupon*(0.1,0.25)(1,0.75){$_{np}$}
	\DGCcoupon*(-1.25,0.25)(-0.1,0.75){$_{(n+2)p}$}
    \DGCcoupon*(-0.25,1)(0.25,1.15){}
    \DGCcoupon*(-0.25,-0.15)(0.25,0){}
	\end{DGCpicture} 
	\ ,
\quad 
\quad
\quad
\Theta\left(~
	\begin{DGCpicture}
	\DGCstrand[thn2](0,0)(0,1)
	\DGCdot*<{0.05}
	\DGCcoupon*(0.1,0.25)(1,0.75){$^{n+2}$}
	\DGCcoupon*(-1,0.25)(-0.1,0.75){$^n$}
    \DGCcoupon*(-0.25,1)(0.25,1.15){}
    \DGCcoupon*(-0.25,-0.15)(0.25,0){}
	\end{DGCpicture} 
~\right)
:=
\begin{DGCpicture}
	\DGCstrand[thk1](0,0)(0,1)[$^p$`$_p$]
	\DGCdot*<{0.05}
	\DGCcoupon*(0.1,0.25)(1.25,0.75){$^{(n+2)p}$}
	\DGCcoupon*(-1,0.25)(-0.1,0.75){$^{np}$}
    \DGCcoupon*(-0.25,1)(0.25,1.15){}
    \DGCcoupon*(-0.25,-0.15)(0.25,0){}
	\end{DGCpicture} 
	\ .
\]
\item[(ii)] On the generating $2$-morphisms of $(\mc{U},\dif_0)$, define
\[
\Theta\left(~
\begin{DGCpicture}
\DGCstrand[thn2](0,0)(0,1.5)
\DGCdot{0.75}
\DGCdot>{1.45}
\end{DGCpicture}
~\right)
:=~~
\begin{DGCpicture}
\DGCstrand[thk1](0,0)(0,1.5)[$^p$`$_p$]\DGCdot>{1.45}
\DGCcoupon(-0.35,0.5)(0.35,1){$_{e_p^p}$}
\end{DGCpicture} \ ,
\quad \quad \quad \quad
\Theta\left(~
\begin{DGCpicture}
\DGCstrand[thn2](0,0)(1.5,1.5)
\DGCdot>{1.45}
\DGCstrand[thn2](1.5,0)(0,1.5)
\DGCdot>{1.45}
\end{DGCpicture}
~\right):=
\begin{DGCpicture}[scale=0.7]
\DGCstrand[thk1](0,0)(1,1.25)[$^p$]
\DGCstrand[thk1](1,1.25)(1,1.75)
\DGCstrand[thk1]/u/(1,1.75)(0,3)/u/[`$_p$]\DGCdot>{2.95}
\DGCstrand[thk1](2,0)(1,1.25)[$^p$]
\DGCstrand[thk1]/u/(1,1.75)(2,3)/u/[`$_p$]
\DGCdot>{2.95}
\end{DGCpicture}
\ ,
\]
\[
\Theta
\left(~
  \begin{DGCpicture}
  \DGCstrand[thn2]/d/(0,0)(1,0)
  \DGCdot*<{-0.25,2}
  \DGCcoupon*(1,-0.5)(1.5,0){$^n$}
  \DGCcoupon*(-0.25,0)(1.25,0.15){}
  \DGCcoupon*(-0.25,-0.65)(1.25,-0.5){}
  \end{DGCpicture}
~\right)
:=
\begin{DGCpicture}[scale =0.5]
\DGCstrand[thk1]/d/(0,0)(0.75,-1)(1.5,0)/u/[$_p$`$_p$]
\DGCdot*<{-0.5,1}[r]{}
\DGCcoupon*(1.5,-0.8)(2,-1.25){$np$}
\end{DGCpicture}
\ ,
\quad \quad \quad
\Theta 
\left(~
  \begin{DGCpicture}
  \DGCstrand[thn2](0,0)(1,0)/d/
  \DGCdot*<{0.25,1}
  \DGCcoupon*(1,0)(1.5,0.5){$^n$}
  \DGCcoupon*(-0.25,0.5)(1.25,0.65){}
  \DGCcoupon*(-0.25,-0.15)(1.25,0){}
  \end{DGCpicture}
~\right)
:=
\begin{DGCpicture}[scale =0.5]
\DGCstrand[thk1]/u/(0,0)(0.75,1)(1.5,0)/d/[$^p$`$^p$]
\DGCdot*<{0.5,2}[r]{}
\DGCcoupon*(1.5,0.7)(1.95,1){$np$}
\end{DGCpicture}
\ .
\]
\end{itemize}
\end{defn}

\begin{rem}
A main goal of this section is to show that, as in Remark \ref{rem-weak-thicking-half}, $\Theta$ induces an isomorphism on the level of slash cohomology of the corresponding $p$-DG categories. As we will see, $\Theta$ preserves the relations of $(\mc{U},\dif_0)$ and $(\dot{\mc{U}}, \dif)$ up to acyclic idempotents, and it induces an equivalence of $\mc{D}(\mc{U},\dif_0)$ onto its image as (enhanced) derived categories (Definition \ref{rmk-enhancement}).
\end{rem}

\begin{lem}\label{lem-theta-onto-cocyle}
The image of the $\Theta$ map on the set of generators consists of $p$-cocycles.
\end{lem}
\begin{proof}
This is clear from the the differential action on the thick generators \eqref{eqn-dif-thick-upwards-splitter}--\eqref{eqn-dif-thick-cup-cap-3}.
\end{proof}

The center $\END_{\mc{U}}(\1_n)$ for each object $\1_n$ inside $\mc{U}$ consists of clockwise or counterclockwise thin ``bubbles'' that may be defined in terms of the $1$-morphism generators:
\begin{equation}
\begin{DGCpicture}
\DGCbubble[thn2](0,1){0.5}
\DGCdot<{1.35,L}
\DGCdot{0.65,R}[r]{$_{l+n-1}$}
\DGCdot*.{1.25,R}[r]{$n$}
\end{DGCpicture} \ , 
\quad \quad \quad
\begin{DGCpicture}
\DGCbubble[thn2](0,1){0.5}
\DGCdot>{1.35,L}
\DGCdot{0.65,R}[r]{$_{l-n-1}$}
\DGCdot*.{1.25,R}[r]{$n$}
\end{DGCpicture} \ .
\end{equation}
Here $l\in \N$ labels the degree of the bubble diagram, which equals $2l$ for these depicted. One may identify clockwise bubbles of degree $2l$ with the $l$th complete symmetric function $h_l$ in infinitely many variables, while the degree-$2l$ counterclockwise bubbles with $(-1)^le_l$ (see \cite[Proposition 8.2]{Lau1}). In this way we identify the center $\END_{\mc{U}}(\1_n)$ with the algebra $\sym$ of symmetric functions. A simple consequence of the previous Lemma is that the thickening functor $\Theta$ induces a quasi-isomorphism of the corresponding $p$-DG centers.

\begin{cor}\label{cor-qis-centers}
For each $n\in \Z$, the $2$-functor $\Theta$ induces a quasi-isomrphism of $p$-DG algebras
\[
\Theta: (\END_{\mc{U}}(\1_n),\dif_0) \lra (\END_{\dot{\mc{U}}}(\1_{np}),\dif).
\]
\end{cor}
\begin{proof}
We show the case when $n\geq 0$, and leave the $n\leq 0$ case to the reader. Under the map $\Theta$, the thin bubble is mapped to the thick bubble carrying a symmetric function as a label:
\[
\Theta
\left(~
\begin{DGCpicture}
\DGCbubble[thn2](0,1){0.5}
\DGCdot<{1.35,L}
\DGCdot{0.65,R}[r]{$_{l+n-1}$}
\DGCdot*.{1.25,R}[r]{$n$}
\end{DGCpicture} 
~\right)=
\begin{DGCpicture}
\DGCbubble[thk1](0,1){0.5}
\DGCdot<{1.35,L}
\DGCdot*.{0.65,L}[r]{$^p$}
\DGCdot*.{1.25,R}[r]{$np$}
\DGCcoupon(0.2,0.85)(0.8,1.15){$_{\pi_l}$}
\end{DGCpicture} \ .
\]
Here the thick strand has thickness $p$, and the polynomial $\pi_l$ in $p$ variables equals the Schur function
\[
\pi_l=\pi_{(p(l+n-1)^p)}(x_1,\dots, x_p)= \left( e_p(x_1,\dots, x_p)^{p}\right)^{l+n-1}=(x_1\cdots x_p)^{p(l+n-1)}.
\]
In \cite[equation (4.29)]{KLMS}, the thick bubbles are identified with Schur functions in infinitely many variables. In their notation, this Schur function corresponds to $\pi_\beta^{\spadesuit}$ whose partition $\beta$ has $p$ equal parts: 
\[
\beta=(p(l+n-1),\dots, p(l+n-1))- ((np-p),\dots, (np-p)) 
= (pl,\dots ,pl),
\]
i.e., this is the Schur function corresponding to the Lima partition $((pl)^p)$. By the Littlewood-Richardson rule and Remark \ref{rmk-infinite-version}, the set $\{\pi_{\mu}|\mu\in LP(p)\}$ consists of $0$-cocycles under the differential \eqref{eqn-dif-on-Schur}, and it forms a set of polynomial generators for $\mH_/(\sym)$. The result follows.
\end{proof}

In fact, a closer check of the proof shows that the effect of $\Theta$ on the center $\END_{\mc{U}}(\1_n)$ agrees with the thickening map $\Theta_0$ on $\sym$ discussed in Remark \ref{rmk-infinite-version}.

\begin{lem}\label{lem-dif-some-splitters}
For the special values $a=b=p$ and $n=kp$, the $p$-differential on $\dot{\mc{U}}$ acts on the elements $\phi^0$ and $\phi_\alpha^i$ ($i=1,\dots, p$) as follows:
\[
\dif\left(~
\begin{DGCpicture}
\DGCstrand[thk1]/ur/(0,0)(1,1)/ur/[$^p$`$_p$]
\DGCdot<{0.05}
\DGCstrand[thk1]/ul/(1,0)(0,1)/ul/[$^p$`$_p$]
\DGCdot>{0.95}
\DGCcoupon*(1,0.25)(1.5,0.75){$kp$}
\end{DGCpicture}
~\right)
= 0,
\quad \quad \quad
\dif\left(~
\begin{DGCpicture}[scale=0.3]
\DGCstrand[thk1]/d/(0,8)(0,6)(2,4)(4,6)(4,8)/u/[$_p$`$_p$]
\DGCdot.{5.5,1}[r]{$_i$}
\DGCdot<{7.5,1}
\DGCdot<{7.3,2}
\DGCstrand[thk1]/d/(0,6)(-2,2)(2,0)(4,-2)/d/[`$^{p-i}$]
\DGCdot<{-1.1}
\DGCstrand[thk1]/d/(4,6)(6,2)(2,0)(0,-2)/d/[`$^{p-i}$]
\DGCdot>{-1.4}
\DGCcoupon(1,3)(3,5){$\pi_\alpha$}
\DGCcoupon*(5,6)(7,8){$kp$}
\end{DGCpicture}
~\right)
=
\sum_{\beta=\alpha+\square}
(\mathrm{C}(\square)+i)
\begin{DGCpicture}[scale=0.3]
\DGCstrand[thk1]/d/(0,8)(0,6)(2,4)(4,6)(4,8)/u/[$_p$`$_p$]
\DGCdot.{5.5,1}[r]{$_i$}
\DGCdot<{7.5,1}
\DGCdot<{7.3,2}
\DGCstrand[thk1]/d/(0,6)(-2,2)(2,0)(4,-2)/d/[`$^{p-i}$]
\DGCdot<{-1.1}
\DGCstrand[thk1]/d/(4,6)(6,2)(2,0)(0,-2)/d/[`$^{p-i}$]
\DGCdot>{-1.4}
\DGCcoupon(1,3)(3,5){$\pi_\beta$}
\DGCcoupon*(5,6)(7,8){$kp$}
\end{DGCpicture}
\ .
\]
\end{lem}
\begin{proof}
This follows from a simple computation using the differential formula on the generators \eqref{eqn-dif-thick-upwards-splitter}--\eqref{eqn-dif-thick-cup-cap-3} and the differential action on Schur polynomials \eqref{eqn-dif-on-Schur}.
\end{proof}

The next result will tell us that the $p$-complex appearing naturally in Lemma \ref{lem-dif-some-splitters} is quite easy to handle. One should compare it with Lemma \ref{lem-p-Lima}, and the proof is similar.

\begin{lem}\label{lem-p-Lima-acyclic-case}
For each $i=\{1,\dots, p\}$, define the $p$-complex $$V_i:=\bigoplus_{\alpha\in P(i,kp-i)}\Bbbk\cdot\pi_{\alpha}, \quad \quad
\dif(\pi_{\alpha}):=\sum_{\beta=\alpha+\square}
(\mathrm{C}(\square)+i)\pi_{\beta}.$$
\begin{itemize}
\item[(i)] If $i=p$, then
\[
\mH_/(V_i)=\mH_{/0}(V_i)\cong 
\bigoplus_{\alpha\in LP(p,(k-1)p)}\Bbbk \cdot \pi_\alpha.
\]
\item[(ii)]If $1\leq i\leq p-1$, then the $p$-complex $V_i$ is contractible. 
\end{itemize}
\end{lem}
\begin{proof}
Part (i) of the Lemma is a special case of Lemma \ref{lem-p-Lima} where we have identified the slash cohomology with the trivial $p$-complex spanned by Lima partitions inside $P(p,kp-p)$.

To show part (ii), we will embed $V_i$ inside the $p$-DG module $\s_i(i)$ (see equation \eqref{eqn-dif-rank-one-mod}) over $\sym_i$ as a $p$-complex direct summand. The latter is acyclic by Lemma \ref{lemma-slash-homology-rank-one-mod}.

Now, define an embedding map
\[
f_i:
V_i\lra \s_i(i)\cong \sym_i\cdot v_i,
\quad \quad
\pi_{\alpha}\mapsto
\pi_\alpha \cdot v_i,
\] 
and a projection map
\[
g_i:\s_i(i) \lra V_i, \quad \quad
\pi_\alpha \cdot v_i \mapsto
\left\{
\begin{array}{cc}
\pi_{\alpha} & \textrm{if}~\alpha\in P(i,kp-i),\\
0 & \textrm{otherwise}.
\end{array}
\right.
\]
A simple computation using the differential action on Schur functions and the Pieri rule
\[
\pi_\alpha\cdot e_1=\sum_{\beta=\alpha+\square}\pi_{\beta}
\]
shows that $g_i\circ f_i=\mathrm{Id}_{U_i}$ as maps of $p$-complexes. The result now follows.
\end{proof}

\begin{rem}
In a similar vein as Remark \ref{rmk-categorifying-binomial-reduction}, Lemma \ref{lem-p-Lima-acyclic-case} gives us a categorification of the quantum binomial specialization
\[
{kp \brack i}_{\mathbb{O}_p}=q^{-i(kp-i)}[\mH_/(V_i)]_{\mathbb{O}_p}= 0
\]  
whenever $i\in \{1,\dots, p-1\}$. Modifying the above proof slightly shows that this is also true for any $i \in \{1,\dots, kp-1\}$ such that $p \nmid i$.
\end{rem}

For the next result, we will crucially utilize the $p$-DG enhancement structure (Remark \ref{rmk-enhancement}) on the derived category $\mc{D}(\dot{\mc{U}},\dif)$.  Namely, for any two objects $n,m\in \Z$ and a finite family  of $1$-morphisms that have the form $\1_n \mc{E}^{(i_1)}\mc{F}^{(j_1)}\dots \mc{E}^{(i_k)}\mc{F}^{(j_k)}\1_m$, where the sequence $(i_1,j_1,\dots, i_k,j_k)$ belongs to a finite index set $I$, the $2$-endomorphism algebra
\[
\END_{\dot{\mc{U}}}\left(\bigoplus_{(i_1,\dots, j_k)\in I}\1_n \mc{E}^{(i_1)}\mc{F}^{(j_1)}\dots \mc{E}^{(i_k)}\mc{F}^{(j_k)}\1_m\right)
\]
has a natural $p$-DG algebra structure coming from the $p$-DG $2$-category $(\dot{\mc{U}},\dif)$. 

\begin{prop}\label{prop-invertiblity}
Let $k\in \N$. The element
\[
{\Phi}_{k}:=
\begin{DGCpicture}
\DGCstrand[thk1]/ur/(0,0)(1.5,1.5)/ur/[$^p$`$_p$]
\DGCdot<{0.05}
\DGCstrand[thk1]/ul/(1.5,0)(0,1.5)/ul/[$^p$`$_p$]
\DGCdot>{1.45}
\DGCcoupon*(1.5,0.55)(2,0.95){$kp$}
\end{DGCpicture}
\bigoplus
\left(\bigoplus_{\alpha\in LP(p,(k-1)p)}
\begin{DGCpicture}
\DGCstrand[thk1]/d/(0,0)(0.75,-1)(1.5,0)/u/[$_p$`$_p$]
\DGCdot*<{-0.5,1}[r]{}
\DGCcoupon*(1.5,-0.8)(2,-1.25){$kp$}
\DGCcoupon(0.4,-1.3)(1.1,-0.7){$\pi_\alpha$}
\end{DGCpicture}\right),
\]
which is regarded as a map of left $p$-DG modules over $\mc{\dot{U}}$:
\[
{\Phi}_{k}: \mc{E}^{(p)}\mc{F}^{(p)}\1_{kp}\lra \mc{F}^{(p)}\mc{E}^{(p)}\1_{kp}\oplus \1_{kp}\{-(k-1)p^2\}\oplus \1_{kp}\{-(k-3)p^2\}\oplus \cdots \oplus \1_{kp}\{(k-1)p^2\} ,
\]
is an invertible morphism in the derived category $\mc{D}(\dot{\mc{U}})$.
\end{prop}
\begin{proof}For a fixed $i\in \{1,\dots, p-1\}$, let us consider the direct sum of the maps
\[
\left(\oplus_{\alpha\in P(i,kp-i)} \phi^i_{\alpha}\right): 
\mc{E}^{(p)}\mc{F}^{(p)}\1_{kp}\lra \mc{F}^{(p-i)}\mc{E}^{(p-i)}\1_{kp}^{\oplus {kp \brack i}} .
\]
By the differential action on the elements $\phi_\alpha^i$ (Lemma \ref{lem-dif-some-splitters}), we see that $(\oplus_{\alpha\in P(i,kp-i)} \phi^i_{\alpha})$ is in fact a map of $p$-DG modules
\[
\left(\oplus_{\alpha\in P(i,kp-i)} \phi^i_{\alpha}\right): 
\mc{E}^{(p)}\mc{F}^{(p)}\1_{kp}\lra \mc{F}^{(p-i)}\mc{E}^{(p-i)}\1_{kp}\otimes V_i^*.
\]
Here $V_i$ is the $p$-complex defined in Lemma \ref{lem-p-Lima-acyclic-case}, and $V_i^*$ denotes the vector space dual of $V_i$ equipped with the dual $p$-complex structure. Since $V_i$ is contractible by that Lemma, the $p$-DG module $ \mc{F}^{(p-i)}\mc{E}^{(p-i)}\1_{kp}\otimes V_i^*$ is an acyclic cofibrant module over $(\dot{\mc{U}},\dif)$. Thus it is isomorphic to zero in the enhanced derived category $\mc{D}(\dot{\mc{U}},\dif)$.

Likewise, one can assemble together $\left(\oplus_{\alpha\in LP_c} \phi^p_{\alpha}\right)$, where $LP_c$ stands for the complement set of Lima partitions $LP_c:=P(p,(k-1)p) \backslash LP(p,(k-1)p)$, into a $p$-DG module map from $\mc{E}^{(p)}\mc{F}^{(p)}\1_{kp}$ onto $\1_{kp}\otimes W$ for some acyclic $p$-complex $W$.

Now, let
\[
\mc{P}:=
 \mc{E}^{(p)}\mc{F}^{(p)}\1_{kp}\oplus \mc{F}^{(p)}\mc{E}^{(p)}\1_{kp}\oplus \1_{kp}\{-(k-1)p^2\}\oplus \1_{kp}\{-(k-3)p^2\}\oplus \cdots \oplus \1_{kp}\{(k-1)p^2\}
\]
and
\[
\mc{A}:= \oplus_{i=1}^{p-1} \mc{F}^{(p-i)}\mc{E}^{(p-i)}\1_{kp}\otimes V_i^* \oplus
\1_{kp}\otimes W,
\]
the module $\mc{A}$ being an acyclic cofibrant $\mc{\dot{U}}$-module since both $V_i^*$ and $W$ are. The natural inclusion of $p$-DG endomorphism algebras
\[
\END_{\dot{\mc{U}}}(\mc{P})\lra \END_{\dot{\mc{U}}}
\left(
\mc{P}\oplus \mc{A}
\right)
\]
is evidently a quasi-isomorphism of $p$-DG algebras. Therefore the element ${\Phi}_k\in \END_{\dot{\mc{U}}}(\mc{P})$ is homotopic to an invertible element in $\END_{\mc{\dot{U}}}(\mc{P}\oplus \mc{A})$ by Sto\v{s}i\'{c}'s formula \eqref{eqn-Stosic-formula-algebraic}. It now follows that the cocycle ${\Phi}_k$ must also be invertible in $\mH_{/0}^0(\END_{\dot{\mc{U}}}(\mc{P}))$. 
\end{proof}

For $\alpha $ inside the set of $p$-Lima partitions $LP(p,(k-1)p)\subset P(p,(k-1)p)$, the Schur functions $\pi_{\alpha}$ are identified with $e_p^{lp}(x_1,\dots,x_p)$ with $l\in \{0,\dots, k-1\}$. Therefore the map $\Phi_k$ in the Proposition naturally comes from thickening (Definition \ref{def-thicken-functor-full}) of the following element in $\mc{U}$:
\begin{equation}\label{eqn-thin-invertible-elt}
{\Phi}_k=\Theta\left(~
\begin{DGCpicture}
\DGCstrand[thn2]/ur/(0,0)(1,1)/ur/
\DGCdot<{0.05}
\DGCstrand[thn2]/ul/(1,0)(0,1)/ul/
\DGCdot>{0.95}
\DGCcoupon*(1,0.25)(1.25,0.75){$k$}
\end{DGCpicture}
+ \sum_{i=0}^{n-1}~
\begin{DGCpicture}[scale=0.6]
\DGCstrand[thn2]/d/(0,0)(1,-1)(2,0)/u/[${}$]
\DGCdot<{-0.5,1}
\DGCdot{-1}[dr]{$_i$}
\DGCcoupon*(2.25,-0.35)(2.75,-1){$k$}
\end{DGCpicture}
~\right) \ .
\end{equation}

\begin{rem}\label{rmk-thick-inverses} By Theorem \ref{thm-catfn-sl2-generic}, although it is not essential to specify the inverse of $\widetilde{\phi}_{k}$ explicitly, the proof of Proposition \ref{prop-invertiblity} actually tells us what the inverse element is inside the slash cohomology algebra $\mH_{/0}(\END_{\dot{\mc{U}}}(\mc{P}\oplus \mc{A}))\cong \mH_{/0}(\END_{\dot{\mc{U}}}(\mc{P}))$. To see this, consider the element
\[
\widetilde{\Phi}_k:={\phi}^0\oplus \left(\oplus_{i=1}^{p}\oplus_{\alpha\in P(i,kp-i)} \phi^i_{\alpha}\right),
\]
which we know is invertible by Sto\v{s}i\'{c}'s formula \eqref{eqn-Stosic-formula-algebraic}, and the inverse is given by
\[
\widetilde{\Psi}_k={\psi}^0\oplus \left(\oplus_{i=1}^{p}\oplus_{\alpha\in P(i,kp-i)} \psi^i_{\alpha}\right).
\]
By adding some zero factors to $\Phi_k$, we can regard both $\Phi_k$ and $\widetilde{\Phi}_k$ as elements in $\END_{\dot{\mc{U}}}(\mc{P}\oplus \mc{A})$. Then the proof of the Proposition tells us that the maps $\Phi_k$ and $\widetilde{\Phi}_k$ differ by some $0$-coboundary element. Therefore the same must be true for the difference between $\Psi_k$ and $\widetilde{\Psi}_k$: the projection of $\widetilde{\Psi}_k$ onto $\END_{\dot{\mc{U}}}(\mc{P})$ should be a cocyle which serves as the inverse for $\Phi_k$. Indeed, diagrammatically, one sees from \cite[equation (5.11)]{KLMS} that the element $\Psi_k$ equals
\[
\Psi_k =
-\
\begin{DGCpicture}
\DGCstrand[thk1]/ur/(0,0)(1.5,1.5)/ur/[$^p$`$_p$]
\DGCdot>{1.45}
\DGCstrand[thk1]/ul/(1.5,0)(0,1.5)/ul/[$^p$`$_p$]
\DGCdot<{0.05}
\DGCcoupon*(1.5,0.95)(2,0.45){$kp$}
\end{DGCpicture}
\bigoplus
\left(
\bigoplus_{i=1}^{k-1}
\bigoplus_{\beta,\gamma \in LP(p,ip)}
\begin{DGCpicture}(-1)^{\frac{p(p+3)}{2}}c_{\beta,\gamma}^{(p^i)}
\DGCstrand[thk1]/u/(0,0)(0.75,0.75)(1.5,0)/d/[$^p$`$^p$]
\DGCdot*>{0.5,1}[r]{}
\DGCcoupon*(1.5,0.8)(2,1.25){$kp$}
\DGCcoupon(0.4,0.45)(1.1,0.95){$\pi_\beta$}
\DGCbubble[thk1](0.7,1.75){0.65}
\DGCdot>{1.75,R}
\DGCcoupon*(0.2,1.5)(1.2,2){$\pi_{\gamma}$}
\end{DGCpicture}
\right)\ .
\]
Here the counterclockwise thick bubble labelled by $\pi_\gamma$ indicates the Schur function $\pi_\gamma$ regarded as an element in $\END_{\dot{\mc{U}}}(\1_{kp})$, and the number $c^{(p^i)}_{\beta, \gamma}$ is the usual Littlewood-Richardson coefficient. When both $\beta$, $\gamma$ are Lima partitions inside $P(p,ip)$, $c^{(p^i)}_{\beta,\gamma}=1$ if they add up to the full $p\times ip$ rectangle, and equals zero otherwise.  It then follows by the differential formula \eqref{eqn-dif-thick-upwards-splitter}--\eqref{eqn-dif-thick-cup-cap-3} that $\Phi_k$ is a cocycle, and it comes from thickening the inverse of the element in equation \eqref{eqn-thin-invertible-elt}.
\end{rem}

\begin{defn}\label{def-full-sl2-subcategory} The $p$-DG $2$-category $\mc{D}_{(p)}(\dot{\mc{U}},\dif)$ is the full $2$-subcategory of $\mc{D}(\dot{\mc{U}},\dif)$ generated by the $1$-morphisms $\1_{(n+2)p}\mc{E}^{(p)}\1_{np}$ and $\1_{(n-2)p}\mc{F}^{(p)}\1_{np}$, where $n\in \Z$. Here both derived categories are considered with their $p$-DG enhanced structure.
\end{defn}

Below, we will be using a derived restriction functor along $\Theta$. We briefly discuss how restriction functors behave in this situation, which is a slight modification of the $p$-DG algebra case. For more details, see \cite[Section 8]{QYHopf}. To do so, notice that $\mc{D}(\mc{\dot{U}},\dif)$ is locally a family of derived categories of $p$-DG algebras with their $p$-DG enhanced structure. 
Now $\Theta$ induces a map of $p$-DG algebras that sends any local identity element of $\mc{U}$, say, $\1_{n+2}\mc{E}\1_n$, onto the idempotent $\1_{(n+2)p}\mc{E}^{(p)}\1_{np}$. Thus we are locally reduced to the following situation: there is a map of $p$-DG algebras $f: (A,\dif_A)\lra (B,\dif_B)$, but $f(1_A)=\epsilon$ is not necessarily the identity of $B$, but an idempotent such that $\dif_B(\epsilon)=0$. In any case we have an adjoint pair of derived functors given by \emph{induction} and \emph{restriction} along $f$:
\begin{equation}\label{eqn-derived-induction}
f^* : \mc{D}(A,\dif_A)\lra \mc{D}(B,\dif_B), \quad f^* (M):=B\epsilon\otimes^{\mathbf{L}}_A M,
\end{equation}
\begin{equation}\label{eqn-derived-restriction}
f_*: \mc{D}(B,\dif_B) \lra \mc{D}(A,\dif_A),\quad f_*(N):=\mathbf{R}\HOM_B(B\epsilon, N)=\epsilon N.
\end{equation}
Thus in this situation, the restriction functor is none other than the usual idempotent truncation on modules. In our particular case, the effect of the restriction functor along the thickening map $\Theta$ is seen through the next result.

\begin{thm}\label{thm-main}
The $p$-DG $2$-functor $\Theta:\mc{U}\lra \mc{\dot{U}}$ induces a derived equivalence from $\mc{D}(\mc{U},\dif_0)$ onto the full subcategory $\mc{D}_{(p)}(\dot{\mc{U}},\dif)$. Consequently the restriction map along $\Theta$
$$
\Theta_*:\mc{D}^c(\mc{\dot{U}},\dif)\lra \mc{D}^c(\mc{U},\dif_0)
$$
is a categorification of the quantum Frobenius map (Definition \ref{def-quantum-Frob}) for $\mathfrak{sl}_2$ over the ring $\mathbb{O}_p$. 
\end{thm}
\begin{proof}
We check that the relations in Definition \ref{def-u-dot} are satisfied on the image of $\Theta$.

Relation (i) follows from Theorem \ref{thm-nil-Hecke-holds} for the map $\Theta^+$, and Relation (ii) is a simple computation using \eqref{eqn-dif-thick-cup-cap-1}--\eqref{eqn-dif-thick-cup-cap-3}. The first half of Relation (iii) ($n\geq 0$ case) holds by Proposition \ref{prop-invertiblity}. Composing the rotational symmetry and vertical reflection of $\dot{\mc{U}}$ (see \cite[Corollary 4.9]{EQ1}) preserves $\dif$ and tells us that the second half of Relation (iii) can be deduced from the first half. Alternatively, one can carry out a similar proof as in Proposition \ref{prop-invertiblity} in this situation. Hence the (enhanced) derived category $\mc{D}_{(p)}(\mc{\dot{U}},\dif)$ is a categorical $\mc{U}$-representation by Theorem \ref{thm-catfn-sl2-generic}, and thus must be a categorical quotient of $\mc{D}(\mc{U},\dif_0)$ under $\Theta$. 

Finally, to show that the functor $\Theta$ is fully-faithful on derived categories, one observes that, modulo 
the central parts of the categories generated by thin and thick bubbles respectively, the endormorphism algebras of $\mc{E}\1_n$, $\mc{F}\1_n$ in $\mc{D}(\mc{U},\dif_0)$ are identified with endormphism algebras of $\mc{E}^{(p)}\1_{np}$ and $\mc{F}^{(p)}\1_{{np}}$ in $\mc{D}(\dot{\mc{U}},\dif)$ under $\Theta$. It remains to show that $\Theta$ restricted to the derived central parts $(\END_{\mc{U}}(\1_n),\dif_0)$ and $(\END_{\dot{\mc{U}}}(\1_{np}),\dif)$ are isomorphisms. The result now follows from Corollary \ref{cor-qis-centers}.
\end{proof}

Recall from Definition \ref{defn-canonical-basis} that we have specialized canonical bases for $\dot{U}_{\mathbb{O}_p}$ and $\dot{U}_\rho$, which are denoted $\mathbb{\dot{B}}_{\mathbb{O}_p}$ and $\mathbb{\dot{B}}_{\rho}$ respectively.

\begin{cor}\label{cor-Frobenius-and-canonical-basis}
Under the derived equivalence $\Theta_*:\mc{D}(\mc{U},\dif_0)\lra \mc{D}_{(p)}(\mc{\dot{U}},\dif)$, the $(\mc{U},\dif_0)$-modules in the collection
$$
\mc{B}_1:=\left\{
\mc{E}^{(a)}\mc{F}^{(b)}\1_n,\mc{F}^{(a)}\mc{E}^{(b)}\1_m|n\leq b-a, m\geq a-b\right\},
$$
which lift the specialized canonical basis $\mathbb{\dot{B}}_{\rho}$, are identified with the $p$-DG modules over $(\mc{\dot{U}},\dif)$ inside
\[
\mc{B}_2:=\left\{
\mc{E}^{(ap)}\mc{F}^{(bp)}\1_{np},\mc{F}^{(ap)}\mc{E}^{(bp)}\1_{mp}|n\leq b-a, m\geq a-b\right\}.
\] 
The symbols of the modules in $\mc{B}_2$ categorify the canonical basis elements $\mathbb{\dot{B}}_{\mathbb{O}_p}$ which lie in the idempotented subalgebra of $\dot{U}_{\mathbb{O}_p}$ generated by $\{E^{(p)}1_{np}, F^{(p)}1_{np}|n\in \Z\}$.
\end{cor}

\begin{proof}
This is a simple consequence of Theorem \ref{thm-main}. Here, one observes that, given a finite subset of $p$-DG modules in $\mc{B}_1$ equipped with the zero differential, the endomorphism $p$-DG algebra for this subset of modules over $(\mc{U},\dif_0)$ is a basic, positively graded $p$-DG algebra with zero differential (here we ignore the symmetric functions). Hence, in the Grothendieck group, the symbols of the modules form a basis. Likewise, a similar property holds for the $p$-DG modules in $\mc{B}_2$ parametrized by the same subset $I$. Locally, the $2$-functor $\Theta$ induces a quasi-isomorphism of the corresponding $p$-DG algebras. 
\end{proof}

\begin{rem}\label{rmk-wrong-way-map}
In fact, the derived induction $2$-functor for $\Theta$ (equation \eqref{eqn-derived-induction})
\[
\Theta^*: \mc{D}(\mc{{U}},\dif_0)\lra \mc{D}(\mc{\dot{U}},\dif)
\]
gives rise to a categorification of a section map for the quantum Frobenius \eqref{eqn-Frobenius-sequence}.

 The reader may wonder why the canonical functor $\Theta$ is going the ``wrong'' way from the $p$-DG $2$-category $(\mc{U},\dif_0)$ to $(\dot{\mc{U}},\dif)$ rather than the ``correct'' direction of the quantum Frobenius map. Such a trick has to be employed in the current stage of the work because we do not yet have a $p$-DG categorical ``quotient'' construction of $\mc{D}(\dot{\mc{U}},\dif)$ by a categorified small quantum group. In \cite{EQ2}, we have realized the categorified small quantum $\dot{u}_{\mathbb{O}_p}$ as the full subcategory in $\mc{D}(\mc{\dot{U}},\dif)$ monoidally generated by strands of thickness $1$. Furthermore, this subcategory is equivalent to $\mc{D}(\mc{U},\dif)$ with a nontrivial differential studied in \cite{KQ,EQ1}. In this sense, the first part of the quantum Frobenius sequence \eqref{eqn-Frobenius-sequence} is categorified by this (derived) embedding
\[
\mc{D}(\mc{U},\dif)\hookrightarrow \mc{D}(\mc{\dot{U}},\dif)
\]
 
In upcoming works, we would like to find the analogue of Drinfeld's ``DG quotient'' construction \cite{DrDG} in the context of hopfological algebra. The current version of Theorem \ref{thm-weak-categorification-half-sl2} and \ref{thm-main} can then be utilized to give an explicit $p$-DG realization of the (derived) quotient category, denoted $\mc{D}(\dot{\mc{U}}/\!/\mc{U})$, as $\mc{D}(\mc{U},\dif_0)$. Then we expect the entire sequence \eqref{eqn-Frobenius-sequence}  to be lifted:
\[
\mc{D}(\mc{U},\dif)\hookrightarrow \mc{D}(\mc{\dot{U}},\dif) \twoheadrightarrow \mc{D}(\dot{\mc{U}}/\!/\mc{U})\cong \mc{D}(\mc{U},\dif_0).
\]
 
The current work has used the fact that the canonical basis is explicitly known for quantum $\mf{sl}_2$ by work of Lusztig \cite{Lus4}.  For instance, this is used in proving that the Grothendieck group of $\mc{D}(\mc{\dot{U}},\dif)$ is isomorphic to $\dot{U}_{\mathbb{O}_p}$, which was done by lifing canonical basis elements to explict $p$-DG modules over $(\dot{\mc{U}},\dif)$ (\cite[Theorem 6.2]{EQ2}). This method will not generalize beyond $\mf{sl}_2$ as canonical bases for higher ranks are unknown, and probably are not even fully classifiable. However, such quotient constructions should be applicable to any simply-laced quantum groups, and the quotient maps from $\mc{D}(\dot{\mc{U}})$ onto 
$\mc{D}(\dot{\mc{U}} /\!/ \mc{U})$
 will be a categorical lifting of the corresponding quantum Frobenius map.
\end{rem}

\addcontentsline{toc}{section}{References}

% ====================================================================
% REFERENCES

\bibliographystyle{alpha}
\bibliography{qy-bib}

% ====================================================================

\vspace{0.1in}

\noindent Y.~Q.: { \sl \small Department of Mathematics, Yale University, New Haven, CT 06511, USA} \newline \noindent {\tt \small email: you.qi@yale.edu}

% ==============================================================================
%
\end{document}